\documentclass[article,oneside,reqno,pdftex]{memoir}

\usepackage{amsmath,amsthm,amssymb}
\usepackage[british]{babel} % use british hypenation
\usepackage[sort,square]{natbib} % finer control of \cite command
\usepackage[bookmarksnumbered=true,colorlinks=true,linkcolor=blue,
            citecolor=blue, urlcolor=blue, pdfpagemode=none, linktocpage=true]{hyperref} 
\usepackage{memhfixc} % fixes hyperref for use with memoir 
\usepackage{diagrams} % used for commutative diagrams

\DeclareFontFamily{OT1}{rsfs}{}
\DeclareFontShape{OT1}{rsfs}{n}{it}{<-> rsfs10}{}
\DeclareMathAlphabet{\curly}{OT1}{rsfs}{n}{it}

\newcommand{\sumskipp}[3]{
\sum_{#1\equiv #2} #3
}

\newcommand{\ind}{\operatorname{ind}}
\newcommand{\sumskip}[1]{\sumskipp{i}{u}{#1}}

\newcommand{\tr}{\operatorname{tr}}

\newcommand{\trw}{\operatorname{Tr}_w}
\newcommand{\Scal}{\operatorname{Scal}}

\newcommand{\orb}{\text{orb}} 
\newcommand{\hcf}{\operatorname{hcf}}

\newcommand{\norm}[1]{\lVert#1\rVert}

\newcommand\C{\mathbb C}
\newcommand\Q{\mathbb Q}
\newcommand\R{\mathbb R}
\newcommand\ZZ{\mathbb Z}

\newcommand\PP{\mathbb P}
\newcommand\WPP{\mathbb W\mathbb P}

\renewcommand\L{\mathcal L}

\renewcommand\O{\mathcal O}
\newcommand\I{\curly I}
\newcommand\m{\mathfrak m}
\newcommand\acts\curvearrowright
\newcommand\take\backslash
\newcommand\into\hookrightarrow
\newcommand\Lfloor{\left\lfloor}
\newcommand\Rfloor{\right\rfloor}

\newcommand{\vol}{\operatorname{vol}}
\newcommand{\Ric}{\operatorname{Ric}}

\renewcommand\_{^{\ }_}

\newcommand{\rk}{\operatorname{rank}}

\newcommand{\Id}{\operatorname{Id}}

\newcommand{\ev}{\operatorname{ev}}
\newcommand{\ord}{\operatorname{ord}}

\newcommand{\Aut}{\operatorname{Aut}}

\newcommand{\Proj}{\operatorname{Proj}\,}
\newcommand{\Spec}{\operatorname{Spec}\,}
\newcommand{\Hilb}{\operatorname{Hilb}}
\newcommand{\FS}{\operatorname{FS}}

\makeatletter \@addtoreset{equation}{chapter} \makeatother

\newtheorem{thm}[equation]{Theorem}
\newtheorem{lem}[equation]{Lemma}

\newtheorem{cor}[equation]{Corollary}
\newtheorem{prop}[equation]{Proposition}

\theoremstyle{definition}

\newtheorem{defn}[equation]{Definition}
\newtheorem{example}[equation]{Example}
\newtheorem{rmk}[equation]{Remark}
\newtheorem*{rmk*}{Remark}

\begin{document}
\pagestyle{myheadings}

\title{\vspace{-2cm} Weighted projective embeddings, stability of orbifolds and constant scalar curvature
K\"ahler metrics}
\author{Julius Ross and Richard Thomas}
\date{}

%%%%%%%%%%%%%%%%%%%%%%%   Documents information for pdf files  %%%%%%%%%%%%%%%%%%%%%%%
\hypersetup{
pdfauthor = {J. Ross and R. P. Thomas},
pdftitle = {Weighted projective embeddings, stability of orbifolds and constant scalar curvature
K\"ahler metrics},
%pdfsubject = {Subject},
pdfkeywords = {MSC  32Q15 53C55 14D20  14L24},
pdfcreator = {LaTeX with hyperref package}
%pdfproducer = {dvips + ps2pdf}
}

%%%%%%%%%%%%%%%%%%%%%%%%%%%%%%%%%%%%%%%%%%%%%%%%%%%%%%%%%%%%%%%%%%%%%%%%

\maketitle
\bibliographystyle{ross}
\setlength{\bibsep}{2pt}
\setlength{\absleftindent}{3mm}
\setlength{\absrightindent}{3mm}
\vspace{-12mm}
\begin{abstract} \vspace{-4mm}
We embed polarised orbifolds with cyclic stabiliser groups into weighted projective space via a weighted form of Kodaira embedding. Dividing by
the (non-reductive) automorphisms of weighted projective space then formally
gives a moduli space of orbifolds. We show how to express this as a \emph{reductive}
quotient and so a GIT problem, thus defining a notion of stability for orbifolds.

We then prove an orbifold version of Donaldson's theorem:
the existence of an orbifold K\"ahler metric
of constant scalar curvature implies K-semistability.

By extending the notion of slope stability to orbifolds we therefore get an explicit
obstruction to the existence of constant scalar curvature orbifold K\"ahler metrics.
We describe the manifold applications of this orbifold result, and show
how many previously known
results (Troyanov, Ghigi-Koll\'ar, Rollin-Singer, the AdS/CFT Sasaki-Einstein obstructions
of Gauntlett-Martelli-Sparks-Yau) fit into this framework.
\end{abstract}

\nobibintoc
\setlength{\cftbeforechapterskip}{0mm} %remove gaps between line of contents
\renewcommand{\tocheadstart}{}  %stops it leaving a gap between lines
\renewcommand{\contentsname}{}  %remove contents title
\tableofcontents* 

%%%%%%%%%%%%%%%%%%%%%%%%%%%%%%%%%%%%%%%%%%%%%%%%%%%%%%%%%%%%%%%%%%%%%%%%
\newpage
\chapter{Introduction}\label{chap:intro}

%% put in an opening paragraph %%
The problem of finding canonical K\"ahler metrics on complex manifolds is central in K\"ahler geometry.  Much of the recent work in this area centres around the conjecture of Yau, Tian and Donaldson that the existence of a constant scalar curvature
K\"ahler (cscK) metric should be equivalent to an algebro-geometric notion of stability.  This notion, called ``K-stability'', should be understood roughly as follows.  Suppose we are looking for such a metric on $X$ whose K\"ahler form lies in the first Chern class of an ample line bundle $L$. Then using sections of $L^k$ one can embed $X$ in a large projective space $\PP^{N_k}$ for $k\gg 0$, and stability is taken in a Geometric Invariant Theory (GIT) sense with respect to the automorphisms of these projective spaces as $k\to\infty$. By the Hilbert-Mumford criterion this in turn can be viewed as a statement about numerical invariants coming from one-parameter degenerations of $X$. The connection with metrics is through the Kempf-Ness theorem, that a stable orbit contains a zero of the moment map. Here this says that a (Chow) stable $X$ can be moved by an automorphism of $\PP^{N_k}$ to be \emph{balanced}, and then the restriction of the Fubini-Study metric on $\PP^{N_k}$
approximates a cscK metric for $k\gg0$.\medskip

In this paper we formulate and study a Yau-Tian-Donaldson correspondence for orbifolds.
On the algebro-geometric side this involves orbifold line bundles, embeddings in weighted
projective space, and a notion of stability for orbifolds. This is related in differential
geometry to orbifold K\"ahler metrics (those which pull back to a genuine K\"ahler metric
upstairs in an orbifold chart; downstairs these are K\"ahler metrics with cone angles
$2\pi/m$ about divisors with stabiliser group $\ZZ/m\ZZ$) and their scalar curvature.
So we restrict to the case of orbifolds with cyclic quotient singularities, but importantly
we do allow the possibility of orbifold structure in codimension one.

Our motivation is not the study of orbifolds per se, but their applications to manifolds.
Orbifold metrics are often the starting point for constructions of metrics on manifolds (see for instance \cite{ghigi_kollar:kaehl_einst_metric_orbif_einst_metric_spher}, and the gluing
construction of \cite{rollin_singer:05:non_minim_scalar_flat_kaehl}) or arise naturally
as quotients of manifolds (for instance quasi-regular Sasaki-Einstein metrics on odd
dimensional manifolds correspond to orbifold K\"ahler-Einstein metrics on the leaf space
of their Reeb vector fields). What first interested us in this subject was the remarkable
work of \cite{jerome:obstr_to_exist_sasak_einst_metric} finding new obstructions to the
existence of Ricci-flat cone metrics on cones over singularities, Sasaki-Einstein metrics on the links of the singularities,
and orbifold K\"ahler-Einstein metrics on the quotient. We wanted to understand their
results in terms of stability. In fact we found that most known results concerning orbifold cscK metrics could be understood through an extension of the ``slope stability" of
\cite{ross_thomas:06:obstr_to_exist_const_scalar,ross_thomas:07:study_hilber_mumfor_criter_for} to orbifolds.

The end product is a theory very similar to that of manifolds, but with a few notable differences requiring new ideas:

%\begin{list}{-}{\setlength{\leftmargin}{8pt}\setlength{\rightmargin}{0pt}}
\begin{itemize}
\item Embedding an orbifold into projective space loses the information of the stabilisers, so instead we show how to embed them faithfully into weighted projective space.  This requires the correct notion of ampleness for an orbi-line bundle $L$, and we are forced to use sections of \emph{more than one} power $L^k$ -- in fact
at least as many as the order of the orbifold (defined in Section \ref{sec:orbibasics}). Then the relevant stability problem is taken not with respect to the full automorphism group of weighted projective space (which is not reductive) but with respect to its reductive part (a product of general linear groups).  This later quotient exactly reflects the ambiguity given by the choice of sections used in the embedding and, it turns out, gives the same moduli problem.
\item By considering the relevant moment maps we define the Fubini-Study K\"ahler metrics on weighted projective space required for stability.  A difference between this and the smooth case is that the curvature of the natural hermitian metric on the hyperplane line bundle is not the Fubini-Study K\"ahler metric, though we prove that the difference becomes negligible asymptotically.
\item A key tool connecting metrics of constant scalar curvature to stability is the asymptotic expansion of the Bergman kernel.   To ensure an expansion on orbifolds similar to that on manifolds we consider not just the sections of $L^k$ but sections of $L^{k+i}$ as $i$ ranges over one or more periods.  Moreover these sections must be taken with appropriate weights to ensure contributions from the orbifold locus add up to give a global expansion.  This is the topic of the companion paper \cite{ross_thomas:weigh_bergm_kernel_orbif}, which also contains a discussion of the exact weights needed.

This choice of weights can also be seen from the moment map framework.  The stability we consider is with respect to a product of unitary groups acting on a weighted projective space, and since the centraliser of this group is large the moment map is only defined up to some arbitrary constants.  These correspond exactly to the weights required for the Bergman kernel expansion, and  the main result of \cite{ross_thomas:weigh_bergm_kernel_orbif}
is that there is a choice of weights (and thus a choice of stability notion) that connects with scalar curvature.

\item The numerical invariants associated to orbifolds and their 1-parameter degenerations
are not polynomial but instead consist of a of polynomial ``Riemann-Roch'' term plus periodic terms coming from the orbifold strata.    The definition of the numerical invariants needed for stability (such as the Futaki invariant) will be made by normalising these periodic terms so they have average zero, and then only using the Riemann-Roch part. Then calculations involving stability become identical to the manifold case, only with the canonical divisor replaced with the orbifold canonical divisor.
\end{itemize}
%\end{list}

After setting up this general framework, our main result is one direction of the Yau-Tian-Donaldson conjecture for orbifolds.

\begin{thm}\label{thm:cscKimpliesstability}
Let $(X,L)$ be a polarised orbifold with cyclic quotient singularities.  If
$c_1(L)$ admits an orbifold K\"ahler metric of constant scalar curvature
then $(X,L)$ is K-semistable.
\end{thm}

Our approach follows the proof given for manifolds by Donaldson in \cite{donaldson:05:lower_bound_calab_funct}. An improvement by Stoppa \cite{stoppa:09:k_stabil_const_scalar_curvat_kahler_manif} says that, as long as one assumes a discrete automorphism group, the existence of a cscK metric actually implies K-stability -- it is natural to ask if this too can be extended to orbifolds. \medskip

Finally we give an orbifold version of the slope semistability of
\cite{ross_thomas:06:obstr_to_exist_const_scalar,ross_thomas:07:study_hilber_mumfor_criter_for},
which we show is implied by orbifold K-semistability. Together with Theorem
\ref{thm:cscKimpliesstability} it gives an obstruction to the existence of orbifold cscK
metrics. We use this to interpret some of the known obstructions in terms of stability,
for instance the work of Troyanov on orbifold Riemann surfaces, Ghigi-Koll\'ar on orbifold projective spaces, and Rollin-Singer on projectivisations of parabolic bundles.  A particularly important class for this theory is Fano orbifolds, where cscK metrics are K\"ahler-Einstein
and equivalent to certain quasi-regular Sasaki-Einstein metrics on odd dimensional manifolds.
In this vein we interpret the Lichnerowicz obstruction of Gauntlett-Martelli-Sparks-Yau in terms of stability.

\section{Extensions}

\noindent\textbf{Non-cyclic orbifolds.}
We have restricted our attention purely to orbifold with cyclic quotient singularities.
It should extend easily to orbifolds whose stabilisers are products of cyclic groups by using several ample (in the sense of Section \ref{oample}) line bundles to embed in a product of weighted projective spaces. To encompass also non-abelian orbifolds one should replace the line bundle with a bundle of higher rank so that the local stabiliser groups can act effectively on the fibre over a fixed point, to give a definition of local
ampleness mirroring \ref{defn:locallyample} in the cyclic case. Then one would hope to embed into weighted Grassmannians. We thank Dror Varolin for this suggestion. \smallskip

\noindent\textbf{More general cone angles and ramifolds.}
It would be nice to extend our results from orbifold K\"ahler
metrics -- which have cone angles of the form $2\pi/p,\ p\in\mathbb N$, along divisors
$D$ -- to metrics with cone angles which are any positive rational multiple of $2\pi$.
It should be possible to study these within the framework of algebro-geometric stability
as well.

The one dimensional local model transverse to $D$ is as follows. In this paper to get
cone angle $2\pi/m$ along $x=0$ we introduce extra local functions $x^{\frac km}$ (by
passing the local $m$-fold cover and working with orbifolds). Therefore to produce cone
angles $2\pi p$ it makes sense to discard the local functions $x,x^2,\ldots x^{p-1}$
and use only $1,x^p,x^{p+1},\ldots$\ . (We could use $1,x^p,x^{2p},\ldots$, i.e. pass
to a $p$-fold quotient instead of an $m$-fold cover, but this would
be less general, producing metrics invariant under $\ZZ/p\ZZ$ rather
than those with this invariance on only the tangent space at $x=0$.)

The map $(x^p,x^{p+1})$ from $\C$ to $\C^2$ is a set-theoretic injection with
image $\{v^p=u^{p+1}\}\subset\C^2$. For very small $x$ (so that $x^{p+1}$ is negligible
compared to $x^p$) it is very close to the $p$-fold
cover $x\mapsto x^p$. More precisely, $\{v^p=u^{p+1}\}$ has
$p$ local branches (interchanged by monodromy) all tangent to the $u$-axis. Going once
round $x=0$ through angle $2\pi$ we go $p$ times round $u=0$ through angle $2\pi p$.
Therefore if we restrict a K\"ahler
metric from $\C^2$ to $\{v^p=u^{p+1}\}$ and pullback to $\C$ we get a smooth K\"ahler
metric away from $x=0$ which has cone angle $2\pi p$ at the origin. Similarly the map
$(x^p,x^{p+1},\ldots,x^{p+k})$ to $\C^{k+1}$ has the same property.

To work globally one has to pick a splitting $H^0(X,L^k)\cong H^0(D,L^k)\oplus
H^0(X,L^k(-D))$ and discard those functions in the second summand which do not vanish
to at least order $p$ along $D$. That is, we take
the obvious map
\begin{equation} \label{ramifold}
X\to\PP\big(H^0(D,L^k)^*\oplus H^0(L^k(-pD))^*\big).
\end{equation}
So instead of Kodaira
embedding, we take an injection which fails to be an embedding in the normal directions
to $D$ just as in the local model above.
(More generally, to get cone angles $2\pi p/q$ one should apply the above description
to an orbifold with $\ZZ/q\ZZ$ stabilisers along $D$ and injections instead into weighted
projective spaces.) One might hope for a relation between balanced injections of $X$
\eqref{ramifold} and cscK metrics with prescribed cone
angles along $D$.  We thank Dmitri Panov for discussions about these ``ramifolds". He
has also pointed out that it is too ambitious to expect the full
theory for manifolds and orbifolds to carry over verbatim to this setting since cscK
metrics with cone angles greater than
$2\pi$ can be non-unique. We hope to return to this in future work.
\smallskip

\noindent\textbf{Zero cone angles, cuspidal metrics and stability of pairs.}
It would be fruitful to consider the limit of large orbifold order.  By this we mean
fixing
the underlying space $X$ and a divisor $D$, then putting $\ZZ/m\ZZ$-stabilisers along
$D$ (as in Section \ref{codim1}) and considering $m\gg 0$.  Then, formally at least, stability in the limit $m\to\infty$ is the same as stability of the underlying space where the numerical invariants are calculated with $K_X$ replaced with $K_X+D$.  This has been studied by Sz\'ekelyhidi \cite{szekelyhidi:07:extrem_metric_k_stabil} under the name of ``relative stability''   of the pair $(X,D)$, which he conjectures to be linked via a Yau-Tian-Donaldson conjecture to the existence of complete ``cuspidal" cscK metrics on $X\take D$. And indeed one can think of orbifold metrics with cone angle $2\pi/m$
along a divisor $D$ as tending (as $m\to\infty$) to a complete metric on $X\take D$ (thanks
to Simon Donaldson and Dmitri Panov for explaining this to us). \smallskip

\noindent\textbf{Pairs.}
In principle this paper gives many other ways of forming moduli spaces of pairs $(X,D)$.
Initially one should take $X$ smooth projective and $D$ a simple normal crossings divisor which is a union
of smooth divisors $D_i$. Labelling the $D_i$ by integers $m_i>0$ satisfying the conditions
of Section \ref{codim1}
we get a natural orbifold structure on $X$ from which we recover $D$ as the locus with
nontrivial stabiliser group. Taking (for instance) the orbifold line bundle produced
by tensoring a polarisation on $X$ by $\O(\sum_iD_i/m_i)$ gives an orbifold line bundle
which is ample in the sense of Section \ref{oample}. Embedding in weighted projective
space as in Section \ref{okod} and dividing the resulting Hilbert scheme by the reductive
group described in Section \ref{chap:moduli} gives a natural GIT problem and notion of
stability.

One should then analyse
which orbischemes appear in the compactification that this produces (in this paper we
mainly study only smooth orbifolds and their cscK metrics). It is quite possible
that the resulting stable pairs will form a new interesting class.
Studying moduli and stability of varieties using GIT fell out of favour, not least because the singularities it allows are not those that arise naturally in birational geometry,
but interesting recent work of Odaka suggests a relationship between the newer notion
of K-stability (rather than Chow stability) and semi-log-canonical singularities. It
is therefore natural to wonder if orbifold K-stability of $(X,D)$ is related to some special types of singularity of pairs (perhaps this is most likely in the $m\to\infty$
limit of the last section). In fact the recent work of Abramovich-Hassett
\cite{Abramovich-Hassett} precisely studies moduli of varieties and pairs using orbischemes,
birational geometry and the  minimal model programme (but not GIT).

An obvious special case is curves with weighted marked points, as studied by Hassett \cite{hassett:03:modul_spaces_weigh_point_stabl_curves} and constructed using GIT by Swinarski \cite{swinarski:git_stabil_weigh_point_curves}.  It is possible that Swinarski's
construction can be simplified by using embeddings in weighted projective space instead
of projective space, and even that his (difficult) stability argument might follow from
the existence of an orbifold cscK metric.

\subsubsection*{Acknowledgements}
We thank Dan Abramovich, Simon Donaldson, Alessandro Ghigi, Hiroshi Iritani, Johan de Jong, Dmitri Panov, Miles Reid, Yann Rollin,  James Sparks, Bal\'azs Szendr\H oi and Dror Varolin for useful conversations. Abramovich and Brendan Hassett have also recently studied moduli of orbifolds and weighted projective embeddings \cite{Abramovich-Hassett}, though from a very different and much more professional point of view. In particular they do not use GIT and are mainly interested in the singularities that occur in the compactification; here we are only concerned with smooth orbifolds for the link to differential geometry.  JR received support from NSF Grant DMS-0700419 and Marie Curie Grant PIRG-GA-2008-230920 and RT held a Royal Society University Research Fellowship while this work was carried out.

\chapter{Orbifold embeddings in weighted projective space}\label{chap:embed}

The proper way to write this paper would be using Deligne-Mumford stacks, but this would
alienate much of its potential readership (as well as the two authors). Most of our
DM stacks are smooth, so there is an elementary description in terms of orbifolds, and
it therefore makes sense to use it. However at points (such as when we consider the central
fibre
of a degeneration of orbifolds) DM stacks, or orbischemes, are unavoidable. At this point
most of the results we need (such as the appropriate version of Riemann-Roch) are only
available in the DM stacks literature. So we adopt the following policy. Where possible
we phrase things in elementary terms using only orbifolds. We state the results we need
in this language, even when the only proofs available are in the DM stacks literature.
Where we do something genuinely new we give proofs using the orbifold language,
even though they of course apply more generally to orbischemes or DM stacks.

\section{Orbibasics}\label{sec:orbibasics}

We sketch some of the basics of the theory of orbifolds and refer the reader to \cite{boyer_galicki:08:sasak_geomet,ghigi_kollar:kaehl_einst_metric_orbif_einst_metric_spher}
for more details.  An orbifold consists of a variety $X$ (either an algebraic variety
or, for us, an analytic space) with only finite quotient singularities that is covered
by orbifold charts of the form $U\to U/G\cong V\subset X$, where $V$ is an open set in
$X$, $U$ is an open set in $\mathbb C^n$ and $G$ is a finite group acting effectively
on $U$. We also insist on a minimality
condition, that the subgroups of $G$ given by the stabilisers of points of $U$ generate
$G$ (otherwise one should make both $U$ and $G$ smaller -- it is important that we are
using the analytic topology here).

The gluing condition on charts is the following. If $V'\subset V$ are open sets
in $X$ with charts $U'/G'\cong V'$ and $U/G\cong V$ then there should exist a monomorphism $G'\into G$ and an injection $U'\into U$ commuting with the given $G'$-action on $U'$ and its action through $G'\into G$ on $U$.

Notice that these injections are \emph{not} in general unique, so the charts do not have
to satisfy a cocycle condition upstairs, though of course they do downstairs where the
open sets $V$ glue to give the variety $X$. That is, the orbifold charts need not glue since an orbifold need not be
a global quotient by a finite group, though we will see in Remark \ref{globCstar} that
they \emph{are} global $\C^*$-quotients under a mild condition.

It follows from the gluing condition that the \emph{order} of a point $x\in X$ -- the size of the stabiliser of any lift of $x$ is any orbifold chart -- is well defined.  The \emph{order of $X$} is defined to be the least common multiple of the order of its points (which is finite if $X$ is compact). The \emph{orbifold locus} is the set of points with nontrivial stabiliser group.

In this paper we will mostly
consider only compact orbifolds with cyclic stabiliser groups, so that each $G$ is always
cyclic.

By an \emph{embedding} $f\colon X\to Y$ of orbifolds we shall mean an embedding
of the underlying spaces of $X$ and $Y$ such that for every $x\in X$ there exist orbifold
charts $U'\to
U'/G\ni x$ and $U\to U/G\ni f(x)$ such that $f$ lifts
to an equivariant embedding $U'\into U$. We say that the orbifold structure on $X$ is
pulled back from that on $Y$. Similarly we get a notion of isomorphism of orbifolds.
\medskip

Given a point in the orbifold locus with stabiliser group $\ZZ/m\ZZ$,
call its preimage in a chart $p$, with maximal ideal $\m_p$. Split its
cotangent space $\m_p/\m_p^2$ into weight spaces under the group action (and use the
fact that the ring of formal power series about that point is $\oplus_iS^i(\m_p/\m_p^2)$)
to see that locally analytically there is a chart $U\to U/(\ZZ/m\ZZ)$ of the form
\begin{equation} \label{localmodel}
(z_1,z_2,\ldots,z_n) \mapsto (z_1^{a_1},z_2^{a_2},\cdots z_k^{a_k},z_{k+1},\ldots,z_n),
\end{equation}
for some integers $a_i$ which divide $m$. We call this an \emph{orbifold point} of type $\frac{1}{m}(\lambda_1,\ldots,\lambda_k)$ if $\zeta\in \ZZ/m\ZZ$ acts\footnote{Here $\lambda_i$
is a multiple of $m/a_i$, of course. We are disobeying Miles Reid and picking the usual identification of $\ZZ/m\ZZ$ with the $m$th roots of unity.} as
$$
\zeta\cdot(z_1,\ldots,z_k) = (\zeta^{\lambda_1}z_1,\ldots,\zeta^{\lambda_k}z_k).
$$

The general principle is that any local object (e.g. a tensor) on an orbifold is defined
to be an invariant object on a local chart (rather than an object downstairs on the underlying
space). So an \emph{orbifold K\"ahler metric} is an invariant K\"ahler metric on $U$
for each orbifold chart $U\to U/G$ which glues: its pullback under an injection $U'\into
U$ of charts above is the corresponding metric on $U'$. Such a metric descends to give
a K\"ahler metric on the underlying space $X$, but with possible singularities along
the orbifold locus.

For instance the standard orbifold K\"ahler metric on $\C/(\ZZ/m\ZZ)$ is given
by $\frac i2 dz\,d\bar z$, where $z$ is the coordinate on $\C$ upstairs and $x=z^m$ is
the coordinate on the scheme theoretic quotient $\C$. Downstairs this takes the
form $\frac i2m^{-2}|x|^{\frac2m-2}dx\,d\bar x$, which is a singular K\"ahler metric
on $\C$. The circumference of the circle of radius $r$ about the origin is easily calculated
to be $2\pi r/m$, so the metric has cone angle $2\pi/m$ at the origin, whereas usual K\"ahler metrics have
cone angle $2\pi$. More generally for any divisor $D$ in the orbifold locus with
stabiliser group $\ZZ/m\ZZ$, orbifold K\"ahler metrics on $X$ have cone angle $2\pi/m$
along $D$. So it is important for us to think of $\C/(\ZZ/m\ZZ)$ as an orbifold,
and not as its scheme theoretic quotient $\C$. 

Even when the stabilisers have codimension
two (so that the orbifold is determined by the underlying variety with quotient singularities,
and one ``can forget" the orbifold structure if only interested in the algebraic or analytic
structure) an orbifold metric is very different from the usual notion of a K\"ahler metric
over the singularities (i.e. one which is locally the restriction of a K\"ahler metric
from an embedding in a smooth ambient space).

\section{Codimension one stabilisers} \label{codim1}
The cyclic orbifolds which will most interest us will be those for which the orbifold
locus has codimension one. These are the orbifolds whose local model \eqref{localmodel}
has coprime weights $a_i$.

Therefore globally the orbifold is described by the pair $(X,\Delta)$, where
\begin{itemize}
\item $X$ is a \emph{smooth} variety,
\item $\Delta$ is a $\mathbb Q$-divisor of the form
$\Delta = \sum_i \left(1-\frac{1}{m_i}\right)D_i$,
\item the $D_i$ are distinct smooth irreducible effective divisors,
\item $D=\sum D_i$ has normal crossings, and
\item the $m_i$ are positive integers such that $m_i$ and $m_j$ are coprime if $D_i$
and $D_j$ intersect.
\end{itemize}
Then the stabiliser group of points in the intersection of several components $D_i$ will
be the product of groups $\mathbb Z/m_i\mathbb Z$, and this is cyclic by the coprimality
assumption.

Here $\Delta$ is the ramification divisor of the orbifold charts; see Example \ref{Kample}
for the expression of this in terms of the orbifold canonical bundle.

Notice that above we are also claiming the converse: that given such a pair $(X,\Delta)$
it is an easy exercise to construct an orbifold with stabiliser groups $\ZZ/m_i\ZZ$ along
the $D_i$, and this is unique. This can be generalised to Deligne-Mumford stacks \cite{Cadman};
we give a global construction in \eqref{Cadman}.\medskip

Orbifolds with codimension one stabilisers were called ``not well formed" in the days when ``we were doing the wrong thing"
(Miles Reid, Alghero 2006). Then orbifolds
were studied as a means to produce schemes, so only the quotient was relevant.
The orbifold locus could be removed, since the quotient is smooth. Hence in much of the
literature (e.g.\ \cite{dolgachev:82:weigh_projec_variet}) the not well formed
case is unfortunately ignored.
\medskip

More generally, \emph{any} orbifold can be dealt with in much the same way: it can be
described by a pair $(X,\Delta)$ just as above, but where $X$ has at worst finite cyclic
quotient singularities. This is the point of view taken by \cite{ghigi_kollar:kaehl_einst_metric_orbif_einst_metric_spher}.

\section{Weighted projective spaces}
The standard source of examples of orbifolds is weighted projective spaces.  A graded
vector space $V=\oplus_iV^i$ is equivalent to a vector space $V$
with a $\C^*$-action, acting on $V^i$ with weight $i$. Throughout this paper $V$ will
always be finite dimensional, with all weights strictly positive. We can
therefore form the associated weighted projective space
$\PP(V):=(V\take\{0\})/\C^*$. This is sometimes denoted
$\PP(\lambda_1,\ldots\lambda_n)$, where $n=\dim V$ and the $\lambda_j$ are
the weights (so the number of $\lambda_j$ that equal $i$ is $\dim V^i$).

Let $x_j,\ j=1,\ldots,n$, be coordinates on $V$ such that $x_j$ has weight
$-\lambda_j$. Then $\PP(V)$ is covered by the orbifold charts
\begin{eqnarray} \label{chart}
\{x_j=1\} &\!\! \cong\C^{n-1} \\
\downarrow\ \ \ \,\nonumber \\ \PP(V).\ \nonumber 
\end{eqnarray}
The $\lambda_j$th roots of unity $\ZZ/\lambda_j\ZZ\subset\C^*$ act trivially
on the $x_j$ coordinate, preserving the above $\C^{n-1}$ slice. The vertical
arrow is the quotient by this $\ZZ/\lambda_j\ZZ$; the generator
$\exp(2\pi i/\lambda_j)\in\C^*$ acting by
\begin{equation} \label{gaction}
(x_i)\mapsto(\exp(2\pi i\lambda_i/\lambda_j)x_i).
\end{equation}
 The order of $\PP(V)$ is the least common multiple of the weights
$\lambda_j$.
%\nocite{dolgachev:82:weigh_projec_variet} %this is needed to get a citation in a footnote.
%Must be another TeX-bug.
If the $\lambda_j$ have highest common factor $\lambda>1$ then $\PP(V)$ has
generic stabilisers:\nocite{dolgachev:82:weigh_projec_variet}
every point is stabilised by the $\lambda$th roots of unity, and we will usually assume that this is not the case, so $\PP(V)$ inherits the structure of an orbifold with cyclic stabiliser groups.  

The orbifold points of $\PP(V)$ are as follows. Each vertex $P_i:=[0,\dots,1,\dots,0]$ is of type $\frac{1}{\lambda_i}(\lambda_1,\ldots,\widehat{\lambda}_i,\ldots,\lambda_N)$.
The general points along the line $P_iP_j$ are orbifold points of type
$\frac1{\hcf(\lambda_i,\lambda_j)}
(\lambda_1,\ldots,\widehat{\lambda}_i,\ldots,\widehat{\lambda}_{j},\ldots,\lambda_N)$,
with similar orbifold types along higher dimensional strata.

Thus if for some $j$ the $\lambda_i,\,i\ne j$, have highest common factor
$\lambda>1$ then $\PP(V)$ is not well formed: it has a divisor of 
orbifold points with stabiliser group containing $\ZZ/\lambda\ZZ$ along $x_j=0$.
Replacing the $\lambda_i,\,i\ne j$, by $\lambda_i/\lambda$ gives a well formed weighted
projective space \cite{dolgachev:82:weigh_projec_variet,fletcher:00:workin_with_weigh_compl_inter}
which is just the underlying variety without the divisor of orbifold points. As discussed
in the last section, it is important for us \emph{not} to mess with the orbifold structure
in this way.

Similarly the map $\PP^{n-1}\to\PP(\lambda_1,\ldots,\lambda_n),\ [x_1,\ldots,x_n]\mapsto
[x_1^{\lambda_1},\ldots,x_n^{\lambda_n}]$ exhibits the underlying variety of weighted
projective space as a global finite quotient of ordinary projective space. Again this
does \emph{not} give the right orbifold structure of \eqref{chart}, so we do not use
it.

\section{Orbifold line bundles and $\Q$-divisors} \label{Qdiv}
Locally an orbifold line bundle is simply an equivariant line bundle on an orbifold chart.
This differs from an ordinary line bundle pulled back from downstairs which satisfies
the property that the $G$-action on the line over any fixed point is trivial. In other
words (the pull back to an orbifold chart of) an ordinary line bundle has a local invariant
trivialisation, which an orbifold line bundle may not. So in general orbifold line bundles
are \emph{not} locally trivial.

To define them globally we need some notation. Suppose that $V_i,V_j,V_k$
are open sets in $X$ with charts $U_i/G_i\cong V_i$, etc. Then by the definition of an
orbifold the overlaps $V_{ij}:=V_i\cap V_j$, etc, also have charts $U_{ij}/G_{ij}\cong
V_{ij}$ and inclusions $U_{ij}\into U_i,\ G_{ij}\into G_i$, etc.

Given local equivariant line bundles $L_i$
over each $U_i$, the gluing (or cocycle) condition to define a global orbifold line bundle
is the following. Pulling back $L_j$ and $L_i$ to $U_{ij}$ (via its inclusions in $U_j,\,U_i$
respectively) there should be isomorphisms $\phi_{ij}$ from the former to the latter,
intertwining the actions of $G_{ij}$. Pulling back
further to $U_{ijk}$ we call this isomorphism $\phi_{ij}\in L_i\otimes L_j^*$ (suppressing the pullback
maps for clarity). The cocycle condition is that over $U_{ijk}$,
$$
\phi_{ij}\phi_{jk}\phi_{ki}\ \in\ L_i\otimes L_j^*\otimes L_j\otimes L_k^*\otimes L_k
\otimes L_i^*
$$
should be precisely the identity element $1$.

The standard example is the orbifold canonical bundle $K_{\orb}$, which is defined to
be $K_U$ on the chart $U$ (with the obvious $G$-action induced from that on $U$) and
which glues automatically. 
\medskip

\begin{example}
Take $X$ a smooth space with a smooth divisor $D$ along which we put $\ZZ/(m\ZZ)$ stabiliser
group to form the orbifold $(X,(1-1/m)D)$. Then the orbifold line bundle $\O\big(\!-\!
\frac1mD\big)$
is easily defined as the ideal sheaf of the reduced pullback of $D$ to any chart. In
this way it glues automatically.

Locally it has generator $z$, a local coordinate upstairs
cutting out the reduced pullback of $D$. But this has weight one under the $\ZZ/m\ZZ$-action;
it is not an invariant section, so does \emph{not} define a section of the orbifold line
bundle downstairs ($z^{km-1}$ times this generator does, for all $k\ge0$). Therefore this orbifold bundle is \emph{not} locally trivial: it is locally the trivial line bundle
with the weight one \emph{nontrivial} $\ZZ/m\ZZ$-action.

Away from $D$, the section which is $z^{-1}$ times by this weight one generator is both
regular and invariant, so can be glued to the trivial line bundle.
In this way one can give an equivalent definition of $\O\big(\!-\!\frac1mD\big)$ via
transition functions, much as in the manifold case.

Taking tensor powers we can form $\O\big(\frac nmD\big)$
for any integer $n$. This is an ordinary line bundle only for $n/m$ an integer. The inclusion
$\O(-\frac1mD)\into\O_X$ defines a canonical section $s_{D/m}$ of $\O(\frac1mD)$ which
in the orbifold chart above looks like $z$ vanishing on $D$.

The pushdown to the underlying manifold $X$ of $\O\big(\frac nmD\big)$ is the ordinary
line bundle given by the round down
\begin{equation} \label{rounddown}
\O\left(\Lfloor\frac nm\Rfloor D\right).
\end{equation}
That is to say that the (invariant) sections of $\O\big(\frac nmD\big)$ are of the form
$s_{D/m}^{\frac nm-\Lfloor\!\frac nm\!\Rfloor}t$, where $t$ is any section of the ordinary
line bundle $\O\big(\!\Lfloor\frac nm\Rfloor\!\big)$ on $X$.
% a local equivariant generator $s^n_{D/m}$ of $\O\big(\frac nmD\big)$
% in an orbifold chart will have weight $-n/m$, so the \emph{invariant} sections (those
% of weight
% 0) are of the form $z^{\frac nm-\Lfloor\!\frac nm\!\Rfloor}fs$, where $f$ is any regular
% function pulled back from downstairs.
% 
% all local \emph{equivariant} sections of $\O\big(\frac nmD\big)$
% in orbifold charts
% are divisible by the local generator $z^{\frac nm-\Lfloor\!\frac nm\!\Rfloor}$ and
% so are \emph{invariant} regular sections of \eqref{rounddown} times by this generator.

Since tensor product does not commute with round down, we lose information by
pushing down to $X$: the natural consequence of orbifold line bundles not being
locally trivial.
\end{example}

More generally on any orbifold given by a pair $(X,\Delta)$ as in Section \ref{codim1},
orbifold line bundles and their sections correspond to $\Q$-divisors such that the denominator
of the coefficient of $D_i$ must divide $m_i$, and any irreducible divisor $D$ not in
the list of $D_i$ must have integral coefficients. The space of global sections of the
orbifold line bundle is the space of sections of the round down. Care must be taken however;
for instance if $D_1$ and $D_2$ have $\ZZ/m\ZZ$-stabilisers along them and $\O(D_1)\cong\O(D_2)$
this certainly does \emph{not} imply that $\O(D_1/m)\cong\O(D_2/m)$. \medskip

The \emph{tautological line bundle} $\O_{\PP(V)}(-1)$ over the weighted projective space
$\PP(V)$ is the orbi-line bundle
over $\PP(V)$ with fibre over $[v]$ the union of the orbit $\C^*.v\subset
V$ and $0\in V$. (Any two elements in a fibre can be written $w_i=t_i.v$
for $t_i\in\C,\
i=1,2$, so we can define the linear structure by $aw_1+bw_2:=(at_1+bt_2).v$.
Ordinarily this is not the linear structure on $V$ and the fibre $\O_{[v]}(-1)\subset
V$ is not a linear subspace.) Over the orbi-chart \eqref{chart} this is the
trivial line
bundle $\C^{n-1}\times\C$ with the weight one $\ZZ/\lambda_j\ZZ$-action
on the line $\C$ times by its action (\ref{gaction}) on $\C^{n-1}$. In other
words the map
\begin{eqnarray} \label{trivi}
\C^{n-1}\times\C &\to& \C^n \\ \nonumber
(x_1,\ldots,\widehat x_j,\ldots,x_n,t) &\mapsto& (t^{\lambda_1}x_1,\ldots,
t^{\lambda_j},\ldots,t^{\lambda_n}x_n)
\end{eqnarray}
becomes $(\ZZ/\lambda_j\ZZ)$-equivariant when we use the action (\ref{gaction}) on $\C^{n-1}$,
the standard weight-one action on $\C$, and the original weighted
$\C^*$-action on $\C^n$. The map (\ref{trivi}) is defined in order
to take the trivialisation
$1$ of $\C$ to the tautological trivialisation of the pullback of the orbit
to the chart (\ref{chart}) (a point of the chart (\ref{chart}) is a point
of its own orbit and so trivialises it).  

Note that Dolgachev \cite{dolgachev:82:weigh_projec_variet} uses the same notation
$\O_{\PP(V)}(-1)$ to denote the push forward of our $\O_{\PP(V)}(-1)$ to the underlying
space, thus rounding down fractional divisors. Therefore $\O_{\PP(V)}(a+b) = \O_{\PP(V)}(a) \otimes \O_{\PP(V)}(b)$ does not hold for his sheaves, but is true almost
by definition for our orbifold line bundles.
\medskip

As a trivial example, consider $\O(k)$ over the weighted projective line $\PP(1,m)$.
The first coordinate $x$ on $\C^2$ has weight one, so restricts to a linear functional
on orbits (the fibres of $\O(-1)$). It therefore defines a section of $\O(1)$ which
vanishes at the orbifold point $x=0$. Since $x$ is the coordinate upstairs in the chart
\eqref{chart} and $x^m$ the coordinate downstairs, this is $\frac1m$ times by a real
manifold point. The coordinate $y$ has weight $m$ on the fibres of $\O(-1)$ so defines
a section of $\O(m)$ which vanishes
at the manifold point $y=0$.

The underlying variety is the projective space on the degree $m$ variables $x^m,y$, i.e.
it is $\PP^1$ with reduced points $0$ and $\infty$ where these two variables vanish.
Thus
$$
\O_{\PP(1,m)}(k)=\O\left(\frac km(0)\right)=\O\left(\Lfloor\frac km\Rfloor(\infty)+
\left(\frac km-\Lfloor\frac km\Rfloor\right)(0)\right).
$$
Similarly on $\PP(a,b)$ with $pa+qb=1$, the underlying variety is the usual Proj of the
graded ring on the degree $ab$ generators $x^b$ and $y^a$. Denote by $0$ and $\infty$
the zeros of $x^b$ and $y^a$ respectively. Then it is a nice exercise to check that the
orbifold line bundle $\O_{\PP(a,b)}(1)$ is isomorphic to
$$
\O\left(\frac pb(0)+\frac qa(\infty)\right),
$$
of degree $\frac1{ab}$.

\section{Orbifold polarisations} \label{oample}
To define orbifold polarisations we need the right notion of ampleness
or positivity. For manifolds (or schemes) this is engineered to ensure that
the global sections of $L$ generate the local ring of functions at each point.  For orbifolds, this requires also a \emph{local} condition on an orbifold
line bundle $L$, as we explain using the simplest example. Consider the orbifold
$\C/(\ZZ/2\ZZ)$ with local coordinate $z$
on $\C$ acted on by $\ZZ/2\ZZ$ via $z\mapsto-z$. Then $x=z^2$ is a local
coordinate
on the quotient thought of as a manifold. Any line bundle pulled back from
the quotient (i.e. which
has trivial $\ZZ/2\ZZ$-action upstairs when considered as a trivial line
bundle there) has invariant sections $\C[x]=\C[z^2]$.
Therefore it sees the quotient only as a
manifold, missing the extra functions of $\sqrt x=z$ that the orbifold
sees. So we do not think of it as locally ample: if we tried to embed using its sections we would ``contract" the stabilisers, leaving us with the underlying manifold.

Conversely the trivial line bundle upstairs with nontrivial $\ZZ/2\ZZ$-action
(acting as $-1$ on the trivialisation) has invariant sections $\sqrt x\,\C[x]=z\C[z^2]$.
Its square has trivial $\ZZ/2\ZZ$-action and has sections $\C[x]=\C[z^2]$
as above. Therefore its sections and those of its powers generate the entire ring of functions
$\C[\sqrt x]=\C[z]$ upstairs, and see the full orbifold structure.

\begin{defn}\label{defn:locallyample}
An orbifold line bundle $L$ over a cyclic orbifold $X$ is \emph{locally ample}
if in an orbifold chart around $x\in X$, the stabiliser group acts
faithfully on the line $L_x$.  We say $L$ is \emph{orbi-ample} if it is both locally ample and globally positive.  (By globally positive here we mean $L^{\ord(X)}$ is ample in the usual sense when thought of as a line bundle on the underlying space of $X$;  from the Kodaira-Baily embedding theorem \cite{baily:57:imbed_v_manif_in_projec_space} one can equivalently ask that $L$ admits a hermitian metric with positive curvature.)

By a \emph{polarised} orbifold we mean a pair $(X,L)$ where $L$ is an orbi-ample line bundle on $X$.
\end{defn}

Note that ordinary line bundles on the underlying space are never ample on genuine orbifolds.   Some care needs to be taken when applying the usual theory to orbi-ample
line bundles.  For instance it is not necessarily the case that the tensor product of locally ample line bundles remain locally ample, but if $L$ is locally ample then so
is $L^{-1}$. One can easily check that $L$ is orbi-ample if and only if $L^k$ is ample for one (or all) $k>0$ coprime to $\ord(X)$.

\begin{example} \label{Kample}
The orbifold canonical bundle $K_{\orb}$ is locally ample along divisors of orbifold
points, but not necessarily at codimension two orbifold points. For instance the quotient
of $\C^2$ by the scalar action of $\pm1$ has trivial canonical bundle, so local ampleness
is not determined in codimension one.

Suppose that $X$ is smooth but with a divisor $D$ with stabiliser group $\ZZ/m\ZZ$.
Locally write $D$ as $x=0$ and pick a chart with coordinate $z$ such that $z^m=x$. Then
the identity $dx=mz^{m-1}dz=mx^{1-\frac1m}dz$ shows that $X$ has
orbifold canonical bundle
$$
K_{\orb}\ =\ K_X+\left(1-\frac1m\right)D\ =\ K_X+\Delta,
$$
where $K_X$ is the canonical divisor of the variety underlying $X$. More generally if
the orbifold locus is a union of divisors $D_i$ with stabiliser groups $\ZZ/m_i\ZZ$ then
$K_{\orb}=K_X+\Delta$, where $\Delta=\sum_i\big(1-\frac1{m_i}\big)D_i$ as in Section
\ref{codim1}.
\end{example}

\begin{example}
The hyperplane bundle $\O_{\PP(V)}(1)$ on any weighted projective space $\PP(V)$ is locally ample, and it is actually orbi-ample since some power is ample \cite[Proposition 1.3.3]{dolgachev:82:weigh_projec_variet} (we shall also show below that it admits a hermitian metric with positive curvature).  The pullback of an orbi-ample bundle along an orbifold embedding is also orbi-ample, and thus any orbifold embedded in weighted projective space admits an orbi-ample line bundle.  If $(X,\Delta)$ is an orbifold, $X$ is smooth and $H$ is an ample divisor on $X$ then the orbifold bundle $H+\Delta$ of Section \ref{codim1} is orbi-ample if and only if $H+\Delta$ is an ample $\mathbb Q$-divisor on $X$.
  \end{example}

\section{Orbifold Kodaira embedding} \label{okod}
Now fix a polarised orbifold $(X,L)$ and $k\gg0$.
Let $i$ run throughout a fixed indexing set ${0,1,\ldots,M}$, where $M\ge\ord(X)$,
and let $V$ be the graded vector space
$$
V=\bigoplus_iV^{k+i}:=\bigoplus_iH^0(L^{k+i})^*.
$$
We give the $i$th summand weight $k+i$. Map $X$ to the weighted projective
space $\PP(V)$ by
\begin{equation} \label{WKE}
\phi_k(x):=\big[\oplus_i\ev^{k+i}_x\big].
\end{equation}
Here we fix a trivialisation of $L_x$ on an orbifold chart, inducing trivialisations
of all powers
$L_x^{k+i}$, and then $\ev^{k+i}_x$ is the element of $H^0(L^{k+i})^*$ which
takes a section $s\in H^0(L^{k+i})$ to $s(x)\in L^{k+i}_x\cong\C$.
The weights are chosen so that a change in trivialisation induces a change
in $\oplus_i\ev^{k+i}_x$ that differs only by the action of $\C^*$ on $V$.

Picking a basis $s_j^{k+i}$ for $H^0(L^{k+i})$, then, the map can be described
by
$$
\phi_k(x)=\big[(s_j^{k+i}(x))_{i,j}\big].
$$
This map is well defined at all points $x$ for which there exists a
global section of some $L^{k+i}$ not vanishing at $x$.

\begin{prop} \label{Kod}
If $(X,L)$ is a polarised orbifold then for $k\gg0$ the map \eqref{WKE}
is an embedding of orbifolds (i.e.\ the orbifold structure on $X$ is pulled back
from that on the weighted projective space $\PP(V)$) and 
\[\phi_k^*\O_{\PP(V)}(1)\cong L.\]
\end{prop}

\begin{proof}
Fix $x\in X$. It has stabiliser group $\ZZ/m\ZZ$ for some $m\ge1$, and a
local orbifold chart $U/(\ZZ/m\ZZ)$. Let $y\in U$ (with maximal ideal $\mathfrak
m_y$) map to $x$, and decompose
$\mathfrak m_y/\mathfrak m_y^2=\oplus_lV^l$ into weight spaces.
Since we have chosen the indexing set for $i$ to range over at least a full period
of length $m$, at least one of the $L^{k+i}_y$ has weight 0 and,
for each $l$, there is at least one $i_l$ in the indexing set such that $L^{k+i}_y
\otimes V^{i_l}$ has weight 0.

Therefore each of these $\ZZ/m\ZZ$-modules has invariant local generators,
defining local sections of the appropriate power of $L$ on $X$. For $k\gg0$
these extend to global sections, by ampleness. (The pushdowns of the powers
of $L$
from the orbifold to the underlying scheme give sheaves which all come from
a finite collection of sheaves tensored by a line bundle. For $k\gg0$
this line bundle becomes very positive, and so eventually has no cohomology.
This value of $k$ can be chosen uniformly for all $y$ by cohomology
vanishing for a \emph{bounded} family of sheaves on a scheme.)

Therefore, trivialising $L$ locally, the sections generate $\O_y$ and $\mathfrak
m_y/\mathfrak m_y^2$, so the pullback of the local functions on $\PP(V)$
(the polynomials in $(x_i)_{i\ne j}$ on the orbifold chart \eqref{chart})
generate the local functions on $U$. It follows that the map is an embedding
for large $k$. \medskip

Invariantly, the map \eqref{WKE} can be described as follows. Any lift $\tilde
x\in L^{-1}_x$
of $x$ is a linear functional on $L_x$. Similarly $\tilde x^{\otimes(k+1)}$
is a linear functional on $L_x^{k+i}$. Composed with the evaluation map
$\ev_x^{k+i}\colon H^0(L^{k+i})\to L^{k+i}_x$ gives
$$
\tilde x^{\otimes(k+1)}\circ\ev_x^{k+i}\,\colon\ H^0(L^{k+i})\to\C.
$$
Therefore
$$
\oplus_i\big(\tilde x^{\otimes(k+1)}\circ\ev_x^{k+i}\big)\,\in\ \bigoplus_iH^0(L^{k+i})^*=V
$$
is a well defined point, with no $\C^*$-scaling ambiguities or choices. In
other words \eqref{WKE} lifts to a natural $\C^*$-equivariant embedding of
the orbi-line
\begin{equation} \label{embedline}
L^{-1}_x\ \hookrightarrow\ \bigoplus_iH^0(L^{k+i})^*
\end{equation}
onto the $\C^*$-orbit over the point \eqref{WKE}.
This makes it clear that under this weighted Kodaira embedding, the
pullback of the $\O_{\PP(V)}(-1)$ orbifold line bundle over $\PP(V)$ is $L^{-1}$.
\end{proof}

\begin{rmk} \label{rmk:pullback}
That $\phi_k^*\O_{\PP(V)}(-1)=L^{-1}$, even though the embedding uses
the sections of $L^k,\ldots,L^{k+M}$ and not those of $L$,
follows from the fact that we give $H^0(L^{k+i})^*$ weight $k+i$.
This might come as a surprise and appear to contradict what we
know about Kodaira embedding for manifolds. For instance, suppose we embed
the \emph{manifold} $\PP^1$ using $\O(2)$. Under the normal Kodaira embedding
we get a conic in $\PP^2=\PP(H^0(\O_{\PP^1}(2)^*)$ such that the pullback
of $\O_{\PP^2}(-1)$ is $\O_{\PP^1}(-2)$.

However, from the above orbifold
perspective, this is \emph{not} an embedding
of $\PP^1$, but of the orbifold $\PP^1/(\ZZ/2\ZZ)$, where the $\ZZ/2\ZZ$-action
is trivial. We see this as follows. At the level of line bundles \eqref{embedline},
it is an
embedding of $\O_{\PP^1}(-1)\big/(\ZZ/2\ZZ)$ into $\O_{\PP^2}(-2)$,
where the $\ZZ/2\ZZ$-action is by $-1$ on each fibre. As a manifold this
quotient is indeed $\O_{\PP^1}(-2)$, but as an orbifold it is instead an
orbifold line bundle over the orbifold $\PP^1/(\ZZ/2\ZZ)$, where the $\ZZ/2\ZZ$-action
is trivial.
\end{rmk}

\begin{rmk}
When we began this project in early 2006 we were using a different, perhaps more
natural, weighted projective embedding. We embedded in the same way in
$$
\PP\Big(\bigoplus_iH^0(L^{ik})^*\Big),
$$
where we give $H^0(L^{ik})^*$ weight $i$ (\emph{not} $ik$). (Notice how this
cures the problem with Veronese embeddings described in Remark \ref{rmk:pullback} above.)
This can also be shown to pull back the orbifold structure of weighted projective space to that of $X$ when $L$ is ample, and to pull $\O(1)$ back to $L^k$.   However the corresponding Bergman kernel turns out not to be relevant to constant scalar curvature orbifold K\"ahler metrics. We learnt about the related alternative embedding \eqref{WKE} from Dan Abramovich; see \cite{Abramovich-Hassett}. The idea of using weighted projective embeddings certainly goes back further to Miles Reid; see for instance \cite{ReidGraded}.
\end{rmk}

\section{OrbiProj} It is similarly simple to write down an orbifold version of the Proj construction, using
the whole graded ring $\oplus_kH^0(L^k)$ at once. Given a finitely generated graded ring $R=\oplus_{k\ge0}R_k$ (not necessarily generated in degree 1!) we can form the scheme
$\Proj R$ in the usual way \cite[Proposition II.2.5]{hartshorne(77):algeb_geomet}. However this loses information (for instance we could throw
away all the graded pieces except the $R_{nk},\,k\gg0$, and get the same result).

We endow $\Proj R$ with an orbischeme structure by describing the orbischeme charts.
Fix a homogeneous element $r\in R_+$ and consider the Zariski-open subset $\Spec R_{(r)}
=(\Proj R)\take\{r=0\}$. (As usual $R_{(r)}$ is the degree zero part of the localised
ring $r^{-1}R$.) Then
$$
\Spec\,\frac R{(r-1)}\ \longrightarrow\ \Spec R_{(r)}
$$
is our orbichart. Here $R/(r-1)$ is the quotient of $R$ (thought of
as a ring and forgetting the grading) by the ideal $(r-1)$. The map from $R_{(r)}$
sets $r$ to $1$. \smallskip

More simply but less invariantly, pick homogeneous generators and relations for the graded
ring $R$. Then $\Proj R$ is embedded in the weighted projective space on the generators,
cut out by the equations defined by the relations.

\medskip
Given a projective scheme $(X,L)$ and a Cartier divisor $D\subset X$, this gives a very
direct way to produce Cadman's $r$th root orbischeme $\big(X,\big(1-\frac1r\big)D\big)$
\cite{Cadman}. This
has underlying scheme $X$ but with stabilisers $\ZZ/r\ZZ$ along $D$, and in the above
notation it is simply
\begin{equation} \label{Cadman}
\left(X,\Big(1-\frac1r\Big)D\right)\ =\ \Proj\ \bigoplus_{k\ge0}H^0\Big(X,\O\left(\Lfloor\frac kr\Rfloor D\right)\otimes L^k\Big).
\end{equation}
The hyperplane line bundle $\O(1)\otimes L^{-1}$ on this Proj is $\O\big(\frac1rD\big)$. Picking generators and relations for the above graded
ring we see the $r$th root orbischeme very concretely, cut out by equations in weighted
projective space.

\begin{rmk} \label{globCstar}
Although orbifolds need not be global quotients by \emph{finite} groups, we see that
polarised orbifolds \emph{are} global quotients of varieties by $\C^*$-actions. In terms
of the
weighted Kodaira embedding of Proposition \ref{Kod}, we take the total space of $L^{-1}$
over $X$, minus the zero section, and divide by the natural $\C^*$-action on the fibres.
Equivalently, we express the orbifold Proj of the graded ring $R$ as the quotient of
$\Spec(R)\take\{0\}$ by the action of $\C^*$ induced by the grading.
\end{rmk}

\section{Orbifold Riemann-Roch}
Suppose that $L$ is an orbifold polarisation on $X$. We will need the asymptotics
of $h^0(L^k)$ for $k\gg0$. These follow from
Kawazaki's orbifold Riemann-Roch theorem \cite{kawasaki:79:rieman_roch_theor_for_compl_v_manif},
or To\"en's for Deligne-Mumford stacks \cite{toen:99:theor_de_rieman_roch_pour},
and some elementary algebra (see for example \cite{reid:87:young_guide_to_canon_singul}
in the well formed case). Alternatively they follow from the weighted Bergman
kernel expansion (see \cite[Corollary 1.12]{ross_thomas:weigh_bergm_kernel_orbif}), or
by embedding
in weighted projective space and taking hyperplane sections in the usual way. The result is that
\begin{equation} \label{orbRR}
h^0(L^k)\ =\ \frac{\int_Xc_1(L)^n}{n!}\,k^n\ -\ \frac{\int_Xc_1(L)^{n-1}.c_1(K_{\orb})}
{2(n-1)!}\,k^{n-1}\ +\ \tilde
o(k^{n-1}).
\end{equation}
Here and in what follows we define $\tilde o(k^{n-1})$ to mean functions of $k$ that
can be written as  $r(k)\delta(k)+O(k^{n-2})$, where $r(k)$ is a polynomial of degree
$n-1$ and $\delta(k)$ is periodic in $k$ with period $m=\ord(X)$ and \emph{average zero}:
$$
\delta(k)=\delta(k+m), \qquad \sum_{u=1}^m \delta(u)=0.
$$
Therefore the average of $\tilde o(k^{n-1})$
over a period is in fact $O(k^{n-2})$, and we think of it as being a lower
order term than the two leading ones of \eqref{orbRR}.
%uniqueness of $h$ in decomposition ?

Here we are also using integration of Chern-Weil forms on orbifolds (or intersection
theory on DM stacks). Of course integration works for orbifolds just as
it does for manifolds; it is defined in local charts,
but then the local integral is divided by the size of the group. It also extends easily
to orbischemes, just as usual integration works on schemes once we weight by local multiplicities.
\medskip

We give a simple example which nonetheless illustrates a number of the issues we have
been considering.

\begin{example}
  Let $\ZZ/m\ZZ$ act on ordinary $\PP^1$ and the tautological line bundle over it by
  $\lambda\cdot [x,y] = [\lambda x,y]$. Then the quotient $X$ is naturally an orbifold
  with $\ZZ/m\ZZ$ stabilisers at the two points $x=0$ and $y=0$. And the quotient of
  $\O(-1)$ is naturally an orbifold line bundle $L_X^{-1}$ over $X$.
  
  However $L_X$ is not locally ample at $x=0$, since the above action is trivial on the
  fibre over $x=0$. So we ``contract" the orbifold structure of $X$ at this point to
  produce another orbifold $Y$ by
  ignoring the stabiliser group at $x=0$ and thinking of it locally as a manifold. Only
  the orbifold point $y=0$ survives, and $L_X$ automatically descends to an ample orbifold
  line bundle $L_Y$ on $Y$, to which orbifold Riemann-Roch \eqref{orbRR} should therefore
  apply.
  
  The sections of $L_Y^k$ (or those of $L_X^k$; they are the same) are the invariant sections of $\O_{\PP^1}(k)$, a basis for which is  $y^k, y^{k-m} x^m,\ldots,y^{k-m\Lfloor\!\frac{k}{m}\!\Rfloor}
x^{m\Lfloor\!\frac{k}{m}\!\Rfloor}$. In particular $Y=\PP\langle x^m,y\rangle=\PP(m,1)$
and $h^0(L^k) = \Lfloor\frac{k}{m}\Rfloor + 1$.

Writing this as $\frac{k}{m} + 1- \frac{m-1}{2m} + \delta(k)$, where $\delta$ is periodic with average zero, we find
$$
h^0(L^k)=\frac{k}{m} -\frac12\left(-2+\left(1-\frac1m\right)\right) + \delta(k)=k\deg L-\frac12\deg K_{\orb}+\delta(k).
$$
Hence, as expected, the single $\ZZ/m\ZZ$-orbifold point of $Y$ adds $1-1/m$ to the degree
of $K_{\orb}$, and the other orbifold point of $X$ does not show up.
\end{example}

\subsection{Equivariant case} \label{ERR}
Fix a polarised orbifold $(X,L)$ as above, but now with a $\C^*$-action on $L$ linearising one on $X$. We need a similar expansion for the weight of a $\C^*$-action on $H^0(L^k)$. Instead of using the full equivariant Riemann-Roch
theorem we follow Donaldson in deducing what we need by using $\PP^1$ to approximate
$B\C^*=\PP^\infty$ and applying the above orbifold Riemann-Roch asymptotics to the total space of the associated bundle over $\PP^1$.

So let $\O_{\PP^1}(1)^*$ denote the principal $\C^*$-bundle over $\PP^1$ given by the complement of the zero-section in $\O(1)$. Form the associated $(X,L)$-bundle
$$
(\mathcal X,\mathcal L):=\O_{\PP^1}(1)^*\times\_{\C^*}(X,L).
$$
Let $\pi\colon\mathcal X\to\PP^1$ denote the projection.
Then it is clear that $\pi_*\mathcal L^k$ is the associated bundle of the $\C^*$-representation
$H^0(X,L^k)$. Splitting the latter into one dimensional weight spaces splits the former
into line bundles. A line with weight $i$ becomes the line bundle $\O(i)$. It follows
that the \emph{total weight} (i.e. the weight of the induced action on the determinant)
of the $\C^*$-action on $H^0(X,L^k)$ is the first Chern class of $\pi_*\mathcal L^k$.
Therefore
$$
w(H^0(X,L^k))=\chi(\PP^1,\pi_*\mathcal L^k)-\rk(\pi_*\mathcal L^k)
=\chi(\mathcal X,\mathcal L^k)-\chi(X,L^k).
$$
In particular orbifold Riemann-Roch on $\mathcal X$ and $X$ show that this has an expansion $b_0k^{n+1}+b_1k^n+\tilde o(k^n)$, where $b_0=\int_{\mathcal X}\frac{c_1(\mathcal L)^n}{(n+1)!}$\,. \medskip

We can express $b_0$ as an integral over $X$ as follows. Take a hermitian metric $h$ on $L$ which is invariant under the action of $S^1\subset\C^*$ and which has positive curvature $2\pi\omega$.
Let $\sigma$ denote the resulting connection 1-form on the principal $S^1$-bundle given by the unit sphere bundle $S(L)$ of $L$.

Differentiating the $S^1$-action gives a vector field $v$ on $S(L)$. Then $\sigma(v)$ is the pullback of a function $H$ on $X$. With respect to the symplectic form $\omega$, this $H$ is a Hamiltonian for the $S^1$-action on $X$.

Write $(\mathcal X,\mathcal L)$ as the associated bundle to the $S^1$-principal bundle $S(\O_{\PP^1}(1))$ as follows,
$$
(\mathcal X,\mathcal L)=S(\O_{\PP^1}(1))\times\_{S^1}(X,L).
$$
The Fubini-Study connection on $\O_{\PP^1}(1)$ and the connection $\sigma$ on $L$ induce natural connections on $\mathcal X\to\PP^1$ and on $\L\to\mathcal X$. In \cite [Section 5.1]{donaldson:05:lower_bound_calab_funct} Donaldson shows that the latter has curvature $H\omega_{FS}+\omega$. (Here $\omega_{FS}$ is pulled back from $\PP^1$, and we think of $\omega$ as a form on $\mathcal X$ by using its natural connection over $\PP^1$ to split its tangent bundle as $T\mathcal X=T\PP^1\oplus TX$.)
Therefore $b_0$ equals
$$
\frac1{(n+1)!}\int_{\mathcal X}(H\omega_{FS}+\omega)^{n+1}=\frac{n+1}{(n+1)!}\int_{\PP^1}\omega_{FS} \int_XH\omega^n=\int_XH\frac{\omega^n}{n!}. 
$$
This proves

\begin{prop}\label{prop:leadingtermweight}
The total weight of the $\C^*$-action on $H^0(L^k)$ is 
\[ w(H^0(X,L^k))=b_0k^{n+1}+b_1k^n+\tilde o(k^n),\quad\text{where }\
b_0=\int_X H\,\frac{\omega^n}{n!}\,.\]
\end{prop}

We will apply this to weighted projective space $X=\PP(V)$ and also to its sub-orbischemes, where the integral on the right must then take into account scheme-theoretic multiplicities and the possibility of generic stabiliser (so if an irreducible component of $X$ has generic stabiliser $\mathbb Z/m\mathbb Z$ then the integral over it is $\frac{1}{m}$ times the integral over the underlying scheme).

Finally we remark that working with $\O_{\PP^2}(1)$ in place of $\O_{\PP^1}(1)$ replaces
the trace of the infinitesimal action on $H^0(X,L^k)$ (i.e. the total weight) by the
trace of the \emph{square} of the infinitesimal action on $H^0(X,L^k)$, proving it equals
\begin{equation} \label{squaredweight}
c_0k^{n+2}+O(k^{n+1}) \quad\text{where}\ c_0=\int_XH^2\frac{\omega^n}{n!}\,.
\end{equation}

\section{Reducing to the reductive quotient}\label{chap:moduli}

To form a moduli space of polarised varieties $(X,L)$ one first embeds $X$ in projective
space $\PP$ with a high power of $L$, thus identifying $X$ with a point of the relevant
Hilbert scheme of subvarieties of $\PP$. It is easy to see that two points of the Hilbert scheme correspond to abstractly isomorphic polarised varieties if and only if they differ
by an automorphism of $\PP$. Therefore a moduli space of varieties can be formed by taking
the GIT quotient of the Hilbert scheme by the special linear group. (Different choices
of linearisations of the action give different notions of stability of varieties.)

By Proposition \ref{Kod} we can now mimic this for polarised orbifolds, first embedding
in a \emph{weighted} projective space $\PP$. The Hilbert scheme of suborbischemes of
$\PP$ has been constructed in \cite{olsson_starr:03:quot_funct_for_delig_mumfor_stack}.
Therefore we are left with the problem of quotienting this by the action of $\Aut(\PP)$.

At first sight this seems difficult because $\Aut(\PP)$ is not reductive.
Classical GIT works only for reductive groups (though a remarkable amount of the theory
has now been pushed through in the nonreductive case \cite{DoranKirwan}).

As a trivial example consider $\PP(1,2)$ embedded by the identity map in itself. The
automorphisms contain a nonreductive piece $\C$ in which $t\in\C$ acts by
\begin{equation} \label{nonred}
[x,y]\ \mapsto\ [x,y+tx^2].
\end{equation}
However this arises because $\PP(1,2)$ has not been Kodaira embedded as described in
Section \ref{okod}. Using \emph{all} sections of $H^0(\O(1))=\langle x\rangle$ and
$H^0(\O(2))=\langle x^2,y\rangle$ (not just $x$ and $y$) we embed instead via
$$
\PP(1,2)\into\PP(1,2,2),\qquad[x,y]\mapsto[x,x^2,y].
$$
Then the nonreductive $\C$ lies in a bigger, reductive subgroup of $\Aut(\PP(1,2,2))$.
Namely \eqref{nonred} can be realised as the restriction to $\PP(1,2)$ of the automorphism
$$
[A,B,C]\mapsto[A,B,C+tB]
$$
lying in the reductive subgroup $SL(H^0(\O(2)))\subset\Aut\PP(1,2,2)$. Of course it can
also be seen as the restriction of $[A,B,C]\mapsto[A,B,C+tA^2]$, another nonreductive
$\C$ subgroup, but the point is that our embedding has a stabiliser in $\Aut(\PP(1,2,2))$,
and this causes the two copies of $\C$ restrict to the same action. 

Having seen an example, the general case is actually simpler. Given a polarised
variety $(X,L)$, pick an isomorphism from $H^0(X,L^{k+i})$ to a fixed vector space $V^{k+i}$.
Then from Section \ref{okod} we get an
embedding of $X$ into $\PP(\oplus_i(V^{i+k})^*)$. This embedding is \emph{normal}
-- the restriction map $H^0(\O_{\PP}(k+i))\to H^0(\O_X(k+i))$ is an isomorphism by construction.
The next result says that the resulting point of the Hilbert scheme of $\PP$ is unique
up to the action of the \emph{reductive} group $\prod_iGL(V^{k+i})$.

\begin{prop}
Two normally embedded orbifolds $X_j\subset\PP(\oplus_i(V^{i+k})^*)$ are
abstractly isomorphic
polarised varieties if and only if there is $g\in\prod_iGL(V^{k+i})$ such that $g.X_1=X_2$.
\end{prop}

\begin{proof}
If the $(X_j,\O_{X_j}(1)))$ are abstractly isomorphic then their spaces of sections $H^0(\O_{X_j}(k+i))$
are isomorphic vector spaces. Under this isomorphism, the two identifications
$H^0(\O_{X_j}(k+i))\cong V^{i+k},\ j=1,2$, therefore differ by an element $g_{k+i}\in GL(V^{k+i})$. Then $g:=\oplus_{i\,}g_{k+i}$ takes $X_1\subset\PP$ to $X_2$.

The converse is of course trivial, needing only the fact that the action of $\prod_iGL(V^{k+i})$
preserves the polarisation $\O_{\PP}(1)$.
\end{proof}

Therefore one can set up a GIT problem to form moduli of orbifolds, just as Mumford did
for varieties.

Firstly one needs Matsusaka's big theorem for orbifolds,
to ensure that for a fixed $k\gg0$, uniform over all smooth polarised orbifolds of the
same topological type, the orbifold line bundles $L^{k+i}$ have
the number of sections predicted by orbifold Riemann-Roch. This follows by pushing down
to the underlying variety, which has only quotient, and so rational, singularities, to
which \cite[Theorem 2.4]{Matsusaka} applies.

We can thus embed them all in the same weighted projective space. Then one can remove
those suborbifolds of weighted projective space whose embedding is non-normal,
since they are easily seen to be unstable for the action of $\prod_iGL(V^{k+i})$. Thus by
the above result, orbits on the Hilbert scheme really corresponds to isomorphism classes
of polarised orbifolds.
Finally one should compactify with orbischemes (or Deligne-Mumford stacks)
to get proper moduli spaces of stable objects. We do not pursue this here as only smooth
orbifolds and their stability are relevant to cscK metrics, but many of the foundations
are worked out in \cite{Abramovich-Hassett}. (Their point of view is slightly
different from ours -- their notion of stability is related to the minimal model programme
rather than GIT, and they form moduli using the machinery of stacks.)

\chapter{Metrics and balanced orbifolds}\label{chap:metrics}

Our next point of business is to generalise the Fubini-Study metric to weighted projective
space.    Anticipating the application we have in mind, fix some $k\ge 0$ and let
$V=\oplus_{i=1}^M V^{k+i}$ be a finite dimensional graded vector space.  By
a metric $|\cdot|_V$ on a $V$ we will mean a hermitian metric
which makes the vector spaces $V^p$ and $V^q$ orthogonal for $p\neq
q$.  Thus a metric on $V$ is simply given by a hermitian metric
$|\cdot|_{V^p}$ on each $V^p$.  By a graded orthonormal basis $\{t_\alpha^p\}$
for $V$ we mean an orthonormal basis $\{t^p_1,\ldots,t^p_{\dim V^p}\}$ for
$V^p$ for each $p=k+1,\ldots,k+M$.

As usual let $\PP(V)$ be the weighted projective space obtained by declaring that $V^{k+i}$
has weight $k+i$.  The unitary group $U:=\prod_i U(V^{k+i})$ acts on $V$ with moment
map
\begin{equation}
\mu\_U(v) =  \frac{1}{2} \left(v\otimes v^* - \bigoplus_i c_i \Id_{V^{k+i}}\right)
\ \in\ \oplus_i\,\mathfrak u(V^{k+i})^*. \label{eq:momentmapu}
\end{equation}
Here the $c_i$ are arbitrary real constants, which we will take to be positive, and $v^*\in
V^*$ is the linear functional corresponding to $v$ under the hermitian inner
product.  Therefore the $U(1)$ action on $V$ which acts on $V^{k+i}$ with weight $k+i$ has moment map $\mu_{U(1)} = \trw \circ\,\mu_U$, where $\trw\colon \mathfrak u^* \to \mathfrak
u(1)^*$ is the projection $\trw(\oplus_i A^i) = \sum_i (k+i) \tr(A^i)$.  Thus if  $v=\oplus_i v_{k+i}$, then
\begin{equation} \label{cdef}
  \mu\_{U(1)}(v) = \frac12\!\left(\sum_i (k+i)|v_{k+i}|^2 - c\right)\!\!,\quad
  \text{where }\ c:=\sum_i (k+i)c_i\dim V^{k+i}.
\end{equation}

\begin{defn}\label{def:FS}
The \emph{Fubini-Study} orbifold K\"ahler metric $\omega_{FS}$ associated
to \mbox{$|\cdot|_V$} is $\frac1c$ times the metric on $\PP(V)$ which results from viewing it as the
symplectic quotient $\mu_{U(1)}^{-1}(0)/U(1)$ and taking the K\"ahler reduction of the
metric $|\cdot|_V$ under the isometric action of $U$.
\end{defn}

This is an orbifold K\"ahler metric: on the orbifold chart
\eqref{chart} it pulls back to a genuine K\"ahler metric on $\C^{n-1}$. In fact it follows
from Lemma \ref{lem:fubinistudyreduction} below that it is the curvature of a hermitian
metric $h_1$ on the orbifold line
bundle $\O_{\PP(V)}(1)$. The dual of this hermitian metric is the one of three natural
candidates
for the name of Fubini-Study metric on $\O_{\PP(V)}(-1)$. A second natural choice $h_2$
is given by $|v|^2_{h_2} = \sum_{i} |v_{k+i}|^{\frac2{k+i}}$ (note that $|v|^2 = \sum_i
|v_{k+i}|^2$
does not scale correctly under the action of $\C^*$ to define a hermitian metric).
However it is the third candidate $h_3=h_{FS}$ below that we choose. It should be noted
that only on an unweighted projective space do all three agree and metrics.  It seems that $h_{FS}$ is a special case of the more general metrics on
line bundles over toric varieties constructed by Batyrev-Tschinkel \cite[Section 2.1]{batyrev_tschinkel:95:ration_point_bound_heigh_compac_anisot_tori}.

%\begin{proof}
%The uniqueness of $\lambda$ is clear as the polynomial $q(t) = \sum_{p=1}^m
%p|v_p|^2 t^p$ has non-negative coefficients and $q(0)=0$.    The formula
%for the moment map now comes from standard theory as follows.   Let $S=\mu_{U(1)}^{-1}(0)$
%so $m=\mu_{U}|_S$.    Given $v\in V$, the definition \eqref{eq:definitionoflambda}
%of $\lambda$ is made so that $\lambda\cdot v=\oplus \lambda^{p}v_p\in S$.
%Thus \eqref{eq:momentmapu} applied to $\lambda\cdot v$ gives \eqref{eq:momementmapweightedprojcspace}.
%\end{proof}

\begin{defn}\label{def:FSfibre}
The \emph{Fubini-Study metric} $h_{FS}$ on $\O_{\PP(V)}(-1)$ is the hermitian
metric defined by setting the points of $\mu_{U(1)}^{-1}(0)$ to have norm
1. Therefore
\begin{equation*}
|v|_{h_{FS}} := \frac{1}{\lambda(v)}\,,
\end{equation*}
where $\lambda(v).v$ is the unique point of $\mu_{U(1)}^{-1}(0)$ in the orbit $(0,\infty).v$.
That is,  by (\ref{cdef}), $\lambda(v)$ is the unique positive real solution to
\begin{equation}
  \sum_i (k+i)\lambda(v)^{2(k+i)} |v_{k+i}|^2 = c.\label{eq:defoflambda}
\end{equation}

We also use $h_{FS}$ to denote the induced metrics on $\O_{\PP(V)}(i)$.\medskip
%We discovered from Yuri Tschinkel that all three metrics ??** already appeared in his
%work \cite{batyrev_tschinkel:96:heigh_zeta_funct_toric_variet,batyrev_tschinkel:98:conjec_for_toric_variet}
%with Batyrev**.

The discrepancy between $\omega_{FS}$ and the curvature $2\pi\omega_{h_{FS}}:=i\partial\overline\partial\log h_{FS}$ of the metric $h_{FS}$ on $\O_{\PP(V)}(1)$ can be
deduced from a result in \cite{biquard_gauduchon:97:hyper_kaehl_metric_cotan_bundl}.

\begin{lem}\label{lem:fubinistudyreduction}
We have
  \begin{equation} \label{eq:reductionmetric2}
  \omega_{FS}=\omega_{h_{FS}} + \frac i{2c}\partial\overline\partial f,
\end{equation}
where $f\colon\PP(V)\to\R$ is the function
\begin{equation} \label{f}
f:=\sum_i\sum_{\alpha} |t^i_{\alpha}|^2_{h_{FS}}.
\end{equation}
Here $\{t_\alpha^i\}$ is a $|\cdot|_V$-orthonormal basis of $V^*$, so each $t_\alpha^i$
defines a section of $\O_{\PP(V)}(i)$, whose pointwise $h_{FS}$-norm is what appears
in \eqref{f}.
\end{lem}

\begin{proof}
Let $p\colon V\take\{0\} \to \PP(V)$ be projection to the quotient. We use \cite[3.1]{biquard_gauduchon:97:hyper_kaehl_metric_cotan_bundl}; in their notation we set  $\chi$
to be the $c$\,th power homomorphism from $S^1$ to itself and shift our moment map by
$\frac c2$ to agree with theirs. The result is that the pullback
of the K\"ahler form produced by symplectic reduction is
\begin{equation}
 p^*(c\,\omega_{FS}) =  \frac i2\partial\overline\partial\,|\lambda(v).v|_V^2 + \frac i{2\pi}\partial\overline\partial \log \lambda(v)^c,
 \label{eq:reductionmetric}
\end{equation}
where $\lambda(v)\in(0,\infty)$ is defined as in \eqref{eq:defoflambda} so that $\lambda(v).v\in
\mu_{U(1)}^{-1}(0)$.

Over an open set of $\PP(V)$ pick a holomorphic section, or multisection, of $p$, lifting
$x$ to $v=v(x)$. Then the curvature of $h_{FS}$ on $\O_{\PP(V)}(-1)$ is $i\partial
\overline\partial\log|v|_{h_{FS}}$, which by Definition \ref{def:FSfibre} is
$i\partial\overline\partial\log\lambda(v)^{-1}$. Therefore the curvature of $\O_{\PP(V)}(1)$
is $i\partial\overline\partial\log\lambda(v)$ and we can rewrite \eqref{eq:reductionmetric}
(divided through by $c$) as
$$
p^*(\omega_{FS}) =  \frac i{2c}\partial\overline\partial\,|\lambda(v).v|_V^2+p^*\omega_{h_{FS}}.
$$
Then at $v\in V\take\{0\}$ lying over a point $x\in\PP(V)$ we calculate $|\lambda(v).v|_V^2$
as
$$
\sum_i|\lambda(v).v|_{V_i}^2 = \sum_i\sum_\alpha|t^i_\alpha(\lambda(v).v)|^2
=\sum_i\sum_\alpha|t^i_\alpha|^2_{h_{FS},x}\,.
$$
The last equality follows from the definition of $h_{FS}$ \eqref{def:FSfibre}, since
$\lambda(v).v$ lies in $\mu_{U(1)}^{-1}(0)$.
\end{proof}

%where $v=\oplus v_p\in V$ and $\lambda(v)\in \mathbb R^+$ is the unique soulution to  
%\begin{equation}
%\sum_i (k+i)\lambda^{2(k+i)}(v) |v_{k+i}|^2 = c.\label{eq:defoflambda}
%\end{equation}
\end{defn}

%\subsection{Balanced Orbifolds}
The restriction of $\mu_U$ (\ref{eq:momentmapu}) to $\mu_{U(1)}^{-1}(0)$
descends to $\PP(V)$ as the moment map $m$ for the induced action of $U/U(1)$ on $\PP(V)$:
  \begin{equation}
    m ([v]) = \frac{1}{2}\bigoplus_i \left(\lambda^{2(k+i)}(v)\,v_{k+i}\otimes
    v_{k+i}^* - c_i \Id_{V^{k+i}}\right),\label{eq:momementmapweightedprojcspace}
  \end{equation}
with $\lambda(v)$ is defined in \eqref{eq:defoflambda}. Integrating this allows us
to define a notion of balanced orbifolds. 
\begin{defn}\label{def:balanced}
Given an orbifold embedding
$X\subset \PP(V)$  define 
\[
M(X) = \int_X m\, \frac{\omega_{FS}^n}{n!}\,,
\]
where $m$ is the moment map from  \eqref{eq:momementmapweightedprojcspace}.     We say that an orbifold $X\subset \PP(V)$ is \emph{balanced} if
  $M(X)=0$.
\end{defn}

%\begin{defn}
%   We say that an orbifold $X\subset \PP(V)$ is \emph{balanced} if
%  $M(X)=0$.
%\end{defn}

\begin{rmk}
  The balanced condition depends on $|\cdot|_V$ and on
  the choice of constants $c_i$.  Later we will choose specific
  constants to ensure a connection with scalar curvature.

Just as in the manifold situation \cite{donaldson(01):scalar_curvat_projec_embed,wang(04):momen_map_futak_invar_stabil_projec_manif}, $M$ is the
moment map for the action of $U/U(1)$ on Olsson and Starr's Hilbert scheme
\cite{olsson_starr:03:quot_funct_for_delig_mumfor_stack} of sub-orbischemes of $\PP(V)$
endowed with its natural $L^2$-symplectic form. To make sense of this statement one can
either work purely formally, make a precise statement at smooth points, or
restrict attention to a single orbit of
Aut$(\PP(V))$; the latter is smooth and all we will need in the application to constant
scalar curvature. For $X\subset\PP(V)$ and $v,w$ sections of $T\PP(V)|_X$ their pairing
with the symplectic form is defined to be
$$
\Omega(v,w):=\int_Xv\lrcorner\left(w\lrcorner\frac{\omega^{n+1}}{(n+1)!}\right).
$$
The moment map calculation
is the following. We let $A=\oplus_iA^{k+i}$ be a graded Hermitian matrix generating the
1-parameter subgroup $\exp(tA)$ of automorphisms of $\PP(V)$, inducing the vector field
$v_A$ on $\PP(V)$. Since $m_A:=\tr(mA)$ is a hamiltonian for $v_A$, we have
$v_A\lrcorner\,\omega=dm_A$. Moving in the Hilbert scheme down a vector field $v$ on $\PP(V)$ we have
\begin{eqnarray}
\left.\frac d{dt}\right|_{t=0}\!\!\tr(M(X)A) &=& \int_X\mathcal L_v\left(m_A
\frac{\omega^n}{n!}\right) \nonumber \\
&=& \int_Xv(m_A)\frac{\omega^n}{n!}+\int_Xm_Ad\left(v\lrcorner\frac{\omega^n}{n!}\right)
\nonumber \\
&=&  \int_X\omega(v,v_A)\frac{\omega^n}{n!}-\int_Xd(m_A)\wedge
\left(v\lrcorner\frac{\omega^n}{n!}\right) \nonumber \\
&=&  \int_X\omega(v,v_A)\frac{\omega^n}{n!}-\int_X(v_A\lrcorner\,\omega)\wedge
\left(v\lrcorner\frac{\omega^n}{n!}\right) \nonumber \\
&=&  \int_Xv\lrcorner\left(v_A\lrcorner\left(\frac{\omega^{n+1}}{(n+1)!}\right)\right)
=\Omega(v,v_A). \label{mmcalc}
\end{eqnarray}
\end{rmk}

To express the balanced condition in terms of sections of line bundles, fix a polarised orbifold $(X,L)$ with cyclic stabiliser groups.  Embed $X$ in weighted projective space
with $k\gg0$ as in Section \ref{okod}: 
\begin{equation}
\phi_k\colon X \into \PP(V)\ \text{ where }\ V = \bigoplus_{i=1}^M H^0(L^{k+i})^*\text{
and } L = \phi_k^* \O_{\PP(V)}(1). \label{eq:orbiembedding}
\end{equation}
A metric $|\cdot|_V$ on $V$ induces by Definition \ref{def:FSfibre} a Fubini-Study metric
on $\O(1)$, and so one on $L$ which we also denote by $h_{FS}$.
The next Lemma expresses the balanced condition in terms of
coordinates on $V$ given by a graded $|\cdot|_V$-orthonormal basis  $\{t_{\alpha}^{i}\}$,
where $t^i_\alpha\in H^0(L^{k+i})$.  To ease notation we write
\[\vol:=\int_X \frac{c_1(L)^n}{n!}\,.\]

\begin{lem}\label{lem:Mincoordinates}
With respect to these coordinates the matrix $M(X)=\oplus_i M^i(X)$  has entries
\[(M^i(X))_{\alpha\beta} = \frac{1}{2}\left(\int_X  (t^i_{\alpha},t^i_{\beta})_{h_{FS}}\,\frac{\omega_{FS}^n}{n!} -c_i\vol \delta_{\alpha\beta}\right).\]
%In particular $X$ is balanced if and only if
%\[\frac{1}{\vol} \int_{X} h_{\alpha\beta}^i \frac{\omega_{FS}^n}{n!} = c_i\delta_{\alpha\beta} \quad \text{for all } i,\alpha,\beta.\]
\end{lem}
\begin{proof}
Given a point $x$ in $X$ let $\tilde{x}\in L_x^{-1}$ be any non-zero lift, and  write $t_{\alpha}^{i}(\tilde{x})$ for the complex number $(\tilde{x}^{\otimes (k+i)},t_\alpha^{i}(x))$. Then
\[ (t_\alpha,t_\beta)_{h_{FS}} = \lambda(\tilde{x})^{2(k+i)} t_\alpha^{i}(\tilde{x}) \overline{t_{\beta}^{i}(\tilde{x})}\]
where $\lambda(\tilde{x})$ is the positive solution to
$\sum_{i} (k+i) \lambda(\tilde{x})^{2(k+i)}\sum_{\alpha} |t_{\alpha}^{i}(\tilde{x})|^2 = c.$  Now the embedding of $X$ maps $x$ to the point with coordinates $[t_{\alpha}^i(x)]$, so in these coordinates $m(x) = \oplus_i m^i(x)$ where 
\[(m^i(x))_{\alpha\beta} = \frac{1}{2} \left(\lambda(\tilde{x})^{2(k+i)} t_{\alpha}^i(\tilde{x}) \overline{t_\beta^i(\tilde{x})}-c_i\delta_{\alpha\beta}\right) =\frac{1}{2}\left( (t^i_\alpha,t^i_\beta)_{h_{FS}} - c_i\delta_{\alpha\beta}\right)\]
and the result follows by integrating over $X$.
\end{proof}

%\chapter{Balanced metrics}
%We now consider balanced metrics on polarised orbifolds, which we will see in the next section connects the balanced condition to metrics of constants scalar curvature.    Suppose $(X,L)$ is a polarised orbifold.  Then for $k\gg 0$ we have an embedding
%$$ \phi_k \colon X \to \WPP(V) \text{ where } V=\oplus_i H^0(L^{k+i})^*.$$
%Recall that the construction of the Fubini-Study metrics depended on some constants $c_i$ which we now assume to all be positive.  
%Let $h$ be a hermitian metric on $L$, and $\omega$ be a K\"ahler metric in $c_1(L)$.  Also set $\vol = \int_X \frac{\omega^n}{n!}$. 
The balanced condition can also be expressed in terms of hermitian metrics
on $L$. Let $\mathcal K(c_1(L))$ denote the orbifold K\"ahler metrics on
$X$ which are $(2\pi)^{-1}$ times the curvature of an orbifold hermitian metric on $L$. Define maps
\begin{diagram}
  \{ \text{hermitian metrics on $L$}\} \times \mathcal K(c_1(L)) &\pile{\rTo^{\quad \operatorname{Hilb}\quad} \\ \lTo_{\FS} }&   \{\text{metrics on } V\} \\
\end{diagram}
as follows:
\begin{itemize}
\item If $|\cdot|_V$ is a metric on $V:=\oplus_iH^0(L^{k+i})^*$ then  
$$FS(|\cdot|_V) = (\phi_k^*h_{FS},\phi_k^* \omega_{FS}),$$
where $h_{FS}$ and $\omega_{FS}$ are Fubini-Study metrics associated to $|\cdot|_V$.
\item If  $h$ is a hermitian metric on $L$ and $\omega$ a K\"ahler metric in  $\mathcal K(c_1(L))$ the metric $\Hilb(h,\omega)$ on $V$ is defined by requiring that for $s\in H^0(L^{k+i})$
\begin{equation}
|s|^2_{\Hilb(h,\omega)} = \frac{1}{c_i\vol} \int_X |s|^2_h\frac{\omega^n}{n!}\,.\label{eq:definitionofHilb}
\end{equation}
Notice this differs from the usual $L^2$-metric by the $c_i\vol$ factors.  Obviously
$V$ and the maps $\Hilb$ and $\FS$ depend on $k$, and this will always be clear from
the context. 
\end{itemize}
%Note that we are not requiring that the curvature of the metric on $L$ is
%$\omega$, since that is not the case for our Fubini-Study metrics.

%\begin{defn}\label{defn:balanced}
%  \begin{itemize}
%  \item We say that the pair
%  $(h,\omega)$ is \emph{balanced} if it is a fixed point of 
%  $FS\circ \Hilb$.  
%\item A metric $|\cdot|$ on $V$ is \emph{balanced} if it is a fixed point of
%  $\Hilb\circ FS$.
%\end{itemize}
%\end{defn}

\begin{defn}\label{defn:balanced} We say that the pair  $(h,\omega)$ is \emph{balanced}
at level $k$ if it is a fixed point of   $\FS\circ \Hilb$.  A metric $|\cdot|_V$ is said to be \emph{balanced} if it is a fixed point of  $\Hilb\circ \FS$.
\end{defn}

\begin{prop}
%Suppose $\phi\colon X\to \PP(V)$ be the? orbifold  embedding.  
A metric $|\cdot|_V$ on $V$ is balanced if and only if $\phi_k\colon X\subset \PP(V)$ is a balanced
orbifold. 
\end{prop}
\begin{proof}
  Given a metric $|\cdot|_V$ let $(h_{FS},\omega_{FS}) = FS(|\cdot|_V)$ and $\{t_{\alpha}^{i}\}$ be a graded $|\cdot|_V$-orthonormal basis for $V$.  Then by Lemma \ref{lem:Mincoordinates}, $M(X)=0$ if and only if 
\begin{equation*}
\frac{1}{c_i\vol}\int_X (t_{\alpha}^{i}, t^{i}_{\beta})_{h_{FS}} \, \frac{\omega_{FS}^n}{n!} = \delta_{\alpha\beta} \quad\text { for all  } i,\alpha,\beta,
\end{equation*}
if and only if $\{t_{\alpha}^{i}\}$ is orthonormal with respect to the $\Hilb(h_{FS},\omega_{FS})$
metric, if and only if it is the same metric as $|\cdot|_V$.
% So $X$ is balanced if and only if $|\cdot|_V = \Hilb(h_{FS},\omega_{FS}) = \Hilb(FS(|\cdot|_V))$.
\end{proof}

Another way to express the balanced condition is through Bergman kernels.
%A pair $(h,\omega)$ also defines an $L^2$-metric by
%\[\|s\|_{L^2} = \int_X |s|_h^2 \frac{\omega^n}{n!}.\]
\begin{defn}\label{def:weightedbergman}
Let $h$ be a hermitian metric on $L$  and $\omega$ be a K\"ahler metric on $X$.  The \emph{weighted Bergman kernel} is the function
  \begin{equation*}
    \label{eq:defweightedbergman}
    B_k=B_k(h,\omega):= \vol \sum_i c_i (k+i) \sum_{\alpha} |s_{\alpha}^i|_{h}^2
  \end{equation*}
  where $\{s_{\alpha}^i\}$ is graded basis of $\oplus_i H^0(L^{k+i})$ that is orthonormal with respect to the $L^2$-metric defined by $(h,\omega)$.  Equivalently
  \begin{equation*}
    \label{eq:defweightedbergman2}
    B_k=  \sum_i (k+i) \sum_{\alpha} |t_{\alpha}^i|_{h}^2
  \end{equation*}
where $\{t_{\alpha}^i\}$ is orthonormal with respect to the $\Hilb(h,\omega)$ metric.
Of course $B_k$ is independent of these choices of basis.
\end{defn}

If $B_k$ is constant over $X$, then we see by integrating over $X$ that this constant
is necessarily $c= \sum_i c_i (k+i) h^0(L^{k+i})$.
In the unweighted case, $B_k$ can be written invariantly in terms of the ratio of the
hermitian metrics $h$ and $h_{FS}$ on $L$. We have the following analogue here.

\begin{prop}\label{prop:balancedimpliesBergmanconstant}
Fix a hermitian metric $h$ on $L$ and a K\"ahler metric $\omega\in \mathcal K(c_1(L))$ and  let $(h_{FS},\omega_{FS}) = FS\circ \Hilb(h,\omega)$.  Then  $h=h_{FS}$ if and only if $B_k(h,\omega)\equiv c$ is constant on $X$.
%$(h,\omega)$ is balanced if and only if the function 
%\[ B_{L^k}(x):=\sum_i c_i (k+i) \sum_{\alpha} |s_{\alpha}^{k+i}(x)|_h^2\]
%where $\{s_{\alpha}^{k+i}\}$
%is constant.
\end{prop}

\begin{proof}
Let $\{t_{\alpha}^{i}\}$ be a graded basis for $\oplus_i H^0(L^{k+i})$ that is orthonormal with respect to the $\Hilb(h,\omega)$-metric.  For $x\in X$ let $\tilde{x}$ be any non-zero lift in $L^{-1}|_x$.  Then
\begin{eqnarray}\label{eq:balancedbergmancalc}
\sum_{i} (k+i) \sum_{\alpha}  |t_\alpha^{i}(x)|^2_{h_{FS}}  &=& \sum_{i} (k+i) \sum_{\alpha}  |(\tilde{x},t_\alpha^{i}(x))|^2 |\tilde{x}|^{-2(k+i)}_{h_{FS}} \nonumber \\
%    &=& \sum_{i} (k+i)  \sum_{\alpha} \frac{ |(\ev_{\tilde{x}}^{k+i},t_\alpha^{k+i})|^2}{|\tilde{x}|^{2(k+i)}_{h_{FS}}}\\
   &=& \sum_{i}  (k+i) \sum_{\alpha} \lambda(\tilde{x})^{2(k+i)}  |(\ev_{\tilde{x}}^{i},t_\alpha^{i})|^2 \nonumber \\
    &=&  \sum_{i} (k+i) \lambda(\tilde{x})^{2{i}} |\ev_{\tilde{x}}^{i}|^2  
    = c,
  \end{eqnarray}
from the definition of dual norms, the fact that $\lambda(\tilde{x}) = |\tilde{x}|_{h_{FS}}^{-1}$ and the defining equation for $\lambda(\tilde{x})$ \eqref{eq:defoflambda}.    Thus if $h=h_{FS}$ then $B_k$ is constant.  Conversely, if $\beta:= h_{FS}/h$ we have
  \begin{equation}
    c= \sum_{i} (k+i) \sum_{\alpha} |t_\alpha^{i}(x)|^2_{h_{FS}} = \sum_{i}^m (k+i) \sum_{\alpha} \beta(x)^{(k+i)} |t_\alpha^{i}(x)|^2_{h}.\label{eq:balanceduniquesolution2}
  \end{equation}
Now note that for fixed $x$, the quantity $u_i=(k+i)\sum_{\alpha}|t_\alpha^{i}(x)|^2_{h}$
is nonnegative for each $i$, so there is a unique positive real solution to the equation $\sum_{i=1} \beta^{2(k+i)}(x) u_i=c$. If $B_k\equiv c$ is constant then $\beta(x)=1$ is one  solution, and thus the unique solution, so $h= h_{FS}$.
\end{proof}
%The proof of the Proposition follows from this as the $\Hilb$ metric is
%a rescaling of the $L^2$-metric.  Specifically let $s_{\alpha}^{k+i} = \frac{1}{\sqrt{c_i
%\vol}} t_{\alpha}^{k+i}$.  Then $s_{\alpha}^{k+i}$ is $L^2$-orthonormal
%if and only if $\{t_{\alpha}^{k+i}\}$ is $Hilb(h,\omega)$ orthonormal. 

% \begin{note}\label{note:scaleFS}
%  The requirement \ref{eq:sumofcisdimV} is not really necessary if we modify the definition of the Fubini-Study metric to be
%$$|\tilde{x}|^2_{FS} = \frac{\dim V}{c\vol} \frac{1}{\lambda(\tilde{x})^2}.$$  Another way is to put this scaling factor in the definition of the map $FS$.  Perhaps Richard has a preference?
%\end{note}

% \begin{note}
%   It may be possible to improve on the statement of the previous lemma if instead of assuming that $\omega$ is the pullback of $\omega_{FS}$ we assume that the pair $(h,\omega)$ is in the image of the map $FS$.   What I need is a statement of the following form:\\

% (Stuck on this!)     Suppose that $X\to \PP(V)$ is Kodaira-embedding and that $|\cdot|_1$ and $|\cdot|_2$ are two metrics on $V$.  For $i=1,2$ let $(h_i,\omega_i) = FS (|\cdot|_i)$.  If $h_1=h_2$ then $\omega_1=\omega_2$.  (in fact I think that if $h_1=h_2$ then $|v|_1 = |v|_2$ holds for all $v\in V$ with $[v]\in X$).  Richard mentioned that this might only hold up to the automorphisms that fix $X$ pointwise....  Notice this statement has nothing to do with $\Hilb$ metrics or Bergman kernels or balanced metrics.
%\end{note}

\chapter{Limits of Fubini-Study metrics}\label{chap:limitsFS}

The connection between constant scalar curvature metrics and stability comes through
the asymptotics of Fubini-Study metrics.  The crucial ingredient is the asymptotics,
as $k\to\infty$, of the weighted Bergman kernel of Definition \ref{def:weightedbergman}:
\[ B_k = \vol\,\sum_i c_i (k+i)  \sum_{\alpha} |s^i_{\alpha}|^2_h.\]
Here $\{s_{\alpha}^i\}$ is a basis of $H^0(L^{k+i})$ that is orthonormal with respect to the $L^2$-metric induced by $h$ and $\omega$.  Ensuring that this is related to scalar curvature requires a particular choice of $c_i$, so for concreteness assume from now on they are chosen by requiring 
\begin{equation} \label{eq:choiceofci}
\sum_i c_i t^i  := (t^{\ord(X)-1} + t^{\ord(X)-2} + \cdots +1)^{p+1}
\end{equation}
for some sufficiently large integer $p$.  We prove in \cite[1.7 and 4.13]{ross_thomas:weigh_bergm_kernel_orbif} that with this choice of $c_i$ there is an asymptotic expansion
\begin{equation}
B_k = b_0 k^{n+1}+b_1 k^{n} + \cdots \quad\mathrm{as}\ k\to\infty \label{eq:expansion}
\end{equation}
for some smooth functions $b_i$.  Taking larger values of $p$ yields a stronger expansion: in fact if $p\ge r+q$ for integers $r,q\ge 0$ then \eqref{eq:expansion} holds up to terms of order $O(k^{n+1-r})$ in the $C^{q}$-norm. By this we mean that
there is a constant $C$ such that for all $k$,
\[ \left|\!\left| B_k - b_0 k^{n+1} - b_1 k^{n} -\cdots - b_{r-1} k^{n+1-(r-1)} \right|\!\right|
\le Ck^{n+1-r},\]
where the norm is the $C^{q}$-norm taken over $X$ in the orbifold sense, with the pointwise
norm of the derivatives measured with respect to the metric defined by
$\omega$.  Moreover the constant $C$ can be taken to be uniform for $(h,\omega)$ in a
compact set.

To achieve what we need in this paper it is sufficient to select $p=5$, so in particular there is a $C^2$-expansion involving the top two terms $b_0$ and $b_1$; however nothing is lost if the reader prefers to take a larger $p$ for simplicity.  Moreover if $2\pi\omega_h$ denotes the curvature $i\partial\overline\partial\log h$ of $h$, the top two coefficients are given by \cite[1.11]{ross_thomas:weigh_bergm_kernel_orbif}
\begin{eqnarray*}
  b_0 &=& \vol\,\frac{\omega_h^n}{\omega^n}\sum_i c_i,  \\
b_ 1 &=& \vol\,\frac{\omega_h^n}{\omega^n} \sum_i c_i\left((n+1)i + \tr_{\omega_h}(\Ric(\omega))-\frac{1}{2}\Scal(\omega_h)\right).
\end{eqnarray*}
In particular if $2\pi \omega$ is in fact the curvature of $h$ this simplifies to
\begin{equation} \label{b0b1}
b_0 = \vol\, \sum_i c_i, \qquad b_ 1 = \vol\,\sum_i c_i\left((n+1)i +\frac{1}{2}\Scal(\omega)\right).
\end{equation}
Observe that in this case the top order term, $b_0$, is constant over $X$.\medskip

Now integrating the expansion over $X$ shows the quantity
$c=\sum_i c_i (k+i) h^0(L^{k+i})$ is polynomial modulo small terms (this is shown directly in Lemma \ref{lem:ignoringperiodic}).  In fact
\begin{equation}
  c= \vol\sum_i c_i\left[ k^{n+1} + \left((n+1)i + \frac{\,\overline{\!S}}{2}\right) k^{n}\right] + O(k^{n-1}),\label{eq:expansionofc}
\end{equation}
where $\overline{\! S}$ denotes the average of the scalar curvature of any K\"ahler metric
in $\mathcal K(c_1(L))$.

Similarly \cite[Remark 4.13]{ross_thomas:weigh_bergm_kernel_orbif} there is also an asymptotic
expansion
\begin{equation} \label{another}
\vol\,\sum_i c_i  \sum_{\alpha} |s^i_{\alpha}|^2_h\ =\ b_0 k^n+b_1' k^{n-1} + \cdots
\end{equation}
for some function $b_1'$, and where $b_0$ is as above. Here the choice of $c_i$ is as above \eqref{eq:choiceofci}, and if $p\ge r+q$ the expansion is in the $C^{q}$-norm up to terms of order $O(k^{n-r})$.\medskip

% The
%next lemma shows how this comes about from the particular choice of $c_i$,
%and illustrates how we ignore the periodic terms our definition of stability \eqref{def:futaki}.

In what follows fix a hermitian metric $h$ on $L$ and K\"ahler metric $\omega\in \mathcal K(c_1(L))$, and let $(h_{FS,k},\omega_{FS,k})$ be the pair $FS\circ \Hilb(h,\omega)$ coming from the embedding $X\subset \PP(\oplus_i H^0(L^{k+i})^*)$.
For embeddings of manifolds in ordinary projective space, the asymptotics of $h/h_{FS,k}$
are those of the Bergman kernel. For orbifolds, the fact that the Fubini-Study
fibre metric is defined implicitly in Definition
\ref{def:FSfibre} means that we have to work harder.

\begin{thm}\label{thm:expansionFS}
Suppose that $2\pi \omega$ is the curvature of $h$.  Then $(h_{FS,k},\omega_{FS,k})$ converges to $(h,\omega)$ as $k$ tends to infinity.  In fact if  $\,\overline{\!S}$ denotes the average of the scalar curvature then
%  \begin{equation}
% \frac{h}{h_{FS,k}}= 1 + \frac{\,\overline{\!S}-\Scal(\omega)}{2k^2} +
% O\left(\frac{1}{k^3}\right).\label{eq:expansionFSfibre}
% \end{equation}
  \begin{equation}
 \frac{h_{FS,k}}{h}= 1 + \frac{\,\overline{\!S}-\Scal(\omega)}{2} k^{-2} +
 O({k^{-3}})\label{eq:expansionFSfibre}
 \end{equation}
in the $C^2$-norm, and
\begin{equation}
 \omega = \omega_{FS,k} + O(k^{-2})\label{eq:limitFSmetric}
\end{equation}
in $C^0$.    In particular the set of Fubini-Study K\"ahler metrics is dense in $\mathcal K(c_1(L))$.
\end{thm}

\begin{rmk}
The Theorem can be generalised to the case that $\omega$ is not
the curvature of $h$, in which case there will be an additional $O(k^{-1})$ term appearing in the expansion of $h_{FS,k}/h$.
\end{rmk}

\begin{proof}[Proof of \eqref{eq:expansionFSfibre}]

The aim is to find an asymptotic expansion of
  \begin{eqnarray*}
\alpha_k := \frac{h_{FS,k}}h\,. 
\end{eqnarray*}
Set $\mathcal B_{r}:= \sum_{\alpha} |t^r_\alpha|_h^2$
where $\{t^r_{\alpha}\}$ is a basis of $H^0(L^r)$ that is orthonormal with respect to the $\Hilb(h,\omega)$-norm from \eqref{eq:definitionofHilb}, so that $B_k=\sum_i (k+i)\mathcal B_{k+i}$.    Then if $0\neq \tilde{x}\in L_x^{-1}$,
  \begin{eqnarray}
    \sum_i (k+i) \alpha_k^{k+i} \mathcal B_{{k+i}} &=& \sum_i (k+i)\|\tilde{x}^{k+i}\|^{-2}_{h_{FS,k}} \sum_{\alpha} |t_{\alpha}^{k+i}(\tilde{x})|_h^2 \nonumber \\
&=&\sum_i (k+i) \|\tilde{x}^{k+i}\|_{h_{FS,k}}^{-2}\sum_{\alpha} \|\tilde{x}\|^2_{\Hilb(h,\omega)} \nonumber \\
&=&c, \label{eq:defofalphak}
\end{eqnarray}
where the second equality uses the fact that the $t_{\alpha}^{k+i}$ are orthonormal, the third inequality comes from the definition of the $FS$-norm \eqref{eq:defoflambda}, and as in \eqref{cdef}, $c=\sum_i c_i (k+i)h^0(L^{k+i})$ is constant over $X$.\medskip

We aim first for an asymptotic expansion of $\alpha_k$ that holds in $C^0$.   Say a sequence $a_k$ of real numbers is of order $\Omega(k^p)$ if there is a $\delta>0$ such that $a_k\ge \delta k^p$ for $p\gg 0$.  A sequence of real-valued  functions $f_k$ on $X$ is of order $\Omega(k^p)$ if there is a $\delta>0$ with $f_k\ge \delta k^p$ uniformly on $X$ for all $p\gg 0$.  \medskip

\noindent\textbf{Step 1:} We show  $\alpha_k = 1 + O(k^{-1})$ in $C^0$.  Observe that from \eqref{b0b1} and \eqref{eq:expansionofc},
\[B_k=\sum_i (k+i) \mathcal B_{k+i} = \vol \sum_i c_i k^{n+1} + O(k^n) = c+ O(k^n).\]
Taking the difference with  \eqref{eq:defofalphak} gives
\[\sum_i (k+i)\Big(\alpha_k^{k+i}-1\Big)\mathcal B_{k+i} = O(k^n),\]
and so
\begin{equation}
(\alpha_k-1)\sum_i (k+i)\Big[1+\alpha_k + \alpha_k^2 + \dots + \alpha_k^{k+i-1}\Big]
\mathcal B_{k+i} = O(k^n).\label{eq:weakboundalpha}
\end{equation}

Now $\alpha_k$ is pointwise positive, so the term in square brackets is at least $1$, and $\sum_i (k+i) \mathcal B_{k+i}= \Omega(k^{n+1})$, so the sum on the left hand side is $\Omega(k^{n+1})$.  Thus $\alpha_k-1=O(k^{-1})$ as claimed.  \medskip

\noindent\textbf{Step 2:} There are positive constants $C_1,C_2$ such that
\begin{eqnarray}
  C_1&\le&\alpha_k^j \quad \text{ for all }\frac{k}{2}\le j\le k.\nonumber\\
 \alpha_k^j &\le& C_2 \quad \text{ for all }0\le j\le k. \label{eq:uniformboundalpha}
\end{eqnarray}
\begin{proof}
As $\alpha_k = 1+ O(k^{-1})$ we have $C_1^2\le \alpha_k^k\le C_2$ for some $C_1\in(0,1),\
C_2>1$ and all $k\gg 0$.  Thus for $j\ge \frac{k}{2}$ we have $\alpha_k^j\ge C_1$ and
for  $j\le k$ we have $\alpha_k^j\le C_2$. 
\end{proof}

Using this we can improve on Step 1 by observing that the term in square brackets in \eqref{eq:weakboundalpha} is of order $\Omega(k)$ since each power of $\alpha_k$ is
nonnegative, and there are at least $k/2$ terms  bounded from below by $C_1$.  Hence
\[ \alpha_k -1 = O(k^{-2}) \quad \text{in } C^0.\]

\noindent\textbf{Step 3:}  
Next define
\begin{equation} \label{betak}
\beta_k = 1 + \frac{\overline{\! S}-\Scal(\omega)}{2}k^{-2}.
\end{equation}
We claim that
\begin{eqnarray}
  \sum_i (k+i) \beta_k^{k+i} \mathcal B_{k+i} = c + O(k^{n-1}) \quad\text{ in }C^0.\label{eq:equationforbeta}
\end{eqnarray}
That is, the $\beta_k$ satisfy an implicit equation very close to the one \eqref{eq:defofalphak}
satisfied by the $\alpha_k$, which we shall use to deduce they are approximately
equal.
\begin{proof}
  Note
\[\beta_k^{k+i} = 1 + \frac{\overline{\! S}-\Scal(\omega)}{2}k^{-1} + O(k^{-2}).\]
So using the asymptotic expansion (\ref{eq:expansion}, \ref{b0b1}) of the weighted
Bergman kernel $B_k=\sum(k+i)\mathcal B_{k+i}$,
\begin{align}
\sum_i&  (k+i) \beta_k^{k+i} \mathcal B_{{k+i}}=\sum_i  (k+i)  \left(1 + \frac{\overline{\! S} -\Scal(\omega)}{2k} + O(k^{-2})\right)\mathcal B_{k+i} \nonumber \\
&=\vol\sum_ic_i \left[ k^{n+1} + \left(\frac{\overline{\! S} - \Scal(\omega)}{2} + (n+1)i + \frac{\Scal(\omega)}{2}\right) k^n\right]+ O(k^{n-1})\nonumber \\
&=  \vol \sum_ic_i\left[  k^{n+1} + \left((n+1)i +  \frac{\overline{\! S}}{2}\right) k^n\right] + O(k^{n-1}) \quad \text{in } C^0,\label{eq:calofbetak}
\end{align}
since $\mathcal B_{k+i} = O(k^n)$ in $C^0$. Comparing with \eqref{eq:expansionofc} proves the claim.
\end{proof}\medskip

\noindent\textbf{Step 4:} 
To simplify notation set
\[\gamma_k:=\alpha_k^{k+i-1} + \alpha_k^{k+i-2} \beta_k + \dots + \beta_k^{k+i-1}.\]
Taking the difference between the implicit equations \eqref{eq:defofalphak} and \eqref{eq:equationforbeta}
for $\alpha_k$ and $\beta_k$ yields
\begin{equation}
  (\alpha_k-\beta_k) \sum_i (k+i) \gamma_k \mathcal B_{k+i} = O(k^{n-1}) \quad\text{in } C^0.\label{eq:differenceab}
\end{equation}
From \eqref{eq:uniformboundalpha} and the definition \eqref{betak} of $\beta_k$ we see
that $\gamma_k=\Omega(k)$. Therefore by \eqref{eq:differenceab},
\begin{equation}
\alpha_k=\beta_k + O(k^{-3})= 1 + \frac{\overline{\! S}-\Scal(\omega)}{2}k^{-2} + O(k^{-3}),
\label{eq:differencealphabeta}
\end{equation}
which is the expansion we wanted at the level of $C^0$-norms.\medskip

\noindent\textbf{Step 5:} To extend this to the $C^2$-norm we actually require an expansion in the $C^0$-norm to higher order (this is because although the pieces of the Bergman kernel $\mathcal B_{k+i}$ are of order $O(k^{n})$, their derivatives $D^p\mathcal B_{k+i}$ are of order $O(k^{n+p})$ \cite[Corollary 4.10]{ross_thomas:weigh_bergm_kernel_orbif}, resulting in a loss of a factor of $k$ for each derivative we take).  To achieve this replace $\beta_k$ with
\[ \beta_k = 1 + \frac{\overline{\! S}-\Scal(\omega)}{2}k^{-2} + \tau_1 k^{-3} + 
\tau_2 k^{-4},\]
where the $\tau_i$ are smooth functions independent of $k$. Then the coefficient of $k^{n-1}$
in \eqref{eq:calofbetak} is $b_0\tau_1+f$, where $f$ is independent of $k$ and the $\tau_i$.
Similarly the coefficient of $k^{n-2}$ is
$b_0\tau_2+g$, where $g$ is independent of $k$ and $\tau_2$. 

So setting $\tau_1=-f/b_0$ and $\tau_2=-g/b_0$ we may assume that the $k^{n-1}$ and $k^{n-2}$ terms in \eqref{eq:calofbetak} vanish. Therefore
%\begin{eqnarray*}
%\sum_i (k+i) \beta_k^{k+i} \mathcal B_{k+i}\ =
%  \left\{\!\!\begin{array}{ll} c+O(k^{n-2}) & \text{in $C^0$, and} \vspace{1mm} \\
%   c+O(k^{n-2+\frac{p}{2}}) & \text{in }C^p \text{ for } p=1,2,\end{array}\right.
%\end{eqnarray*}
\begin{equation}
\sum_i (k+i) \beta_k^{k+i} \mathcal B_{k+i}= c+O(k^{n-3+p})  \text{ in } C^p \text{ for } p=0,1,2,
\end{equation}
where we have used  $\mathcal B_{k+i} = O(k^{n+p})$ in $C^p$ in
place of the original argument using $\mathcal B_{k+i}=O(k^n)$ in $C^0$.   Thus
%\begin{equation}
%  (\alpha_k-\beta_k) \sum_i (k+i) \gamma_k \mathcal B_{k+i}\ =
%  \left\{\!\!\begin{array}{ll} O(k^{n-2}) & \text{in $C^0$, and} \vspace{1mm} \\
%   O(k^{n-1}) & \text{in }C^2.\end{array}\right. \label{eq:differenceab2}
%\end{equation}
\begin{equation}
  (\alpha_k-\beta_k) \sum_i (k+i) \gamma_k \mathcal B_{k+i}=O(k^{n-3+p})  \text{ in } C^p,\ p=0,1,2.\label{eq:differenceab2}
\end{equation}
In particular, $\alpha_k = \beta_k + O(k^{-5})$ in $C^0$.\medskip

Now to bound $D\alpha_k$, differentiate \eqref{eq:defofalphak} to get
\begin{equation} \label{dalphak}
D{\alpha_k} \sum_i (k+i)^2 \alpha_k^{k+i-1} \mathcal B_{k+i} = - \sum_i (k+i) \alpha_k^{k+i} D\mathcal B_{k+i}.
\end{equation}
Since  powers of $\alpha_k$ are bounded above uniformly \eqref{eq:uniformboundalpha} and $D\mathcal B_{k+i} = O(k^{n+1})$,  the sum on the right hand side is $O(k^{n+2})$.  On the other hand, using the lower bound in \eqref{eq:uniformboundalpha}, the sum on the left hand side is of order $\Omega(k^{n+2})$, and hence $D\alpha_k = O(1)$.  We claim that $D\gamma_k=O(k^2)$.  In fact both $\alpha_k^j$ and $\beta_k^j$ are uniformly bounded from above for all $k$ and all $j\le k+i$.  Thus if $u+v\le k+i$,
\begin{equation} \label{oslo}
D(\alpha_k^u \beta_k^v) = u\alpha_k^{u-1}\beta_k^v D{\alpha_k} + v\alpha_k^u\beta_k^{v-1}D\beta_k
= O(k),
\end{equation}
since $D\alpha_k=O(1)$ and $D\beta_k = O(k^{-2})$.  Thus $D\gamma_k$ is a sum of $O(k)$ terms each of order $O(k)$ and so $D\gamma_k = O(k^2)$ as claimed.  \medskip

So we know $\gamma_k\mathcal B_{k+i} = O(k^{n+1})$ and $D(\gamma_k \mathcal B_{k+i}) = O(k^{n+2})$.  Differentiating the $p=1$ statement of \eqref{eq:differenceab2} and using $\gamma_k = \Omega(k)$ yields
$$
   (D\alpha_k - D\beta_k)\Omega(k^{n+2})=-(\alpha_k-\beta_k) \sum_i (k+i) D(\gamma_k
   \mathcal B_{k+i}) + O(k^{n-2}) = O(k^{n-2})
$$
as $\alpha_k-\beta_k = O(k^{-5})$.   Hence  $D\alpha_k = D\beta_k +  O(k^{-4})$, and thus we have $\alpha_k-\beta_k = O(k^{-4})$ in $C^1$. In particular $D\alpha_k=O(k^{-2})$. \medskip

A similar argument applies to the second derivative.   Differentiating \eqref{dalphak}
yields 
\begin{eqnarray*}
(D^2 \alpha_k) \Omega(k^{n+2})\!\!&=&\!\!-2\sum_i (k+i)^2\alpha_k^{k+i-1} D\alpha_k D\mathcal B_{k+i} - \sum_i (k+i) \alpha_k^{k+i} D^2\mathcal B_{k+i}\\
&&-\sum_i (k+i)^2(k+i-1)\alpha_k^{k+i-2} (D{\alpha_k})^2 \mathcal B_{k+i}\\
\end{eqnarray*}
which is $O(k^{n+3})$. Thus $D^2\alpha_k = O(k)$.  If $u+v\le k+i$ then
\[D^2(\alpha_k^u\beta_k^v)  = u\alpha_k^{u-1}\beta_k^v D^2\alpha_k + v\alpha_k^u\beta_k^{v-1} D^2\beta_k + O(k^{-2})\]
since $D\alpha_k$ and $D\beta_k$ are both $O(k^{-2})$.   Therefore $D^2(\alpha_k^u\beta_k^v)=O(k^2)$ which implies that $D^2\gamma_k = O(k^3)$ and hence $D^2(\gamma_k \mathcal B_{k+i}) = O(k^{n+3})$.   

Now taking the second derivative of the $p=2$ statement in \eqref{eq:differenceab2},
\begin{eqnarray*}
(D^2\alpha_k - D^2 \beta_k) \Omega(k^{n+2})\!\!&=&\!\!-(\alpha_k-\beta_k) \sum_i (k+i)D^2(\gamma_k \mathcal B_{k+i}) \\
&-&\!\!\!\!\!2(D\alpha_k -D\beta_k) \sum_i (k+i) D(\gamma_k \mathcal B_{k+i}) + O(k^{n-1}).
\end{eqnarray*}
Since $\alpha_k-\beta_k = O(k^{-5})$,   $D\alpha_k-D\beta_k = O(k^{-4})$ and $D(\gamma_k \mathcal B_{k+i}) = O(k^{n+2})$, this is $O(k^{n-1})$.    Hence $D^2\alpha_k = D^2\beta_k + O(k^{-3})$ as required.
\end{proof}

\begin{proof}[Proof of \eqref{eq:limitFSmetric}]
From Lemma \ref{lem:fubinistudyreduction} we have $\omega_{FS,k} = \omega_{h_{FS,k}} + \frac{i}{2c}\partial\overline\partial f_k$, where
\[ f_k = \vol\,\sum_i c_i \sum_{\alpha} |s^i_{\alpha}|_{h_{FS,k}}^2\]
and the $\{s^i_{\alpha}\}$ is a graded basis of $\oplus_i H^0(L^{k+i})$ that is orthonormal with respect to the $L^2$-norm defined by $(h,\omega)$. (So $t^i_\alpha:=\sqrt{c_i\vol\,}
s^i_\alpha$ is an orthonormal basis with respect to the $\Hilb(h,\omega)$ metric.)

Applying $\partial\overline\partial\log$ to \eqref{eq:expansionFSfibre} shows that $\omega_{h_{FS,k}}=\omega_h+O(k^{-2})=\omega+O(k^{-2})$
in $C^0$, so $\omega_{FS,k} = \omega + \frac{i}{2c}\partial\overline\partial f_k+O(k^{-2})$.
So since $c$ is of order $\Omega(k^{n+1})$, to prove \eqref{eq:limitFSmetric} it
will be sufficient to show that $f_k$ is constant on $X$ to $O(k^{n-1})$ in $C^2$-norm.  Applying the expansion \eqref{eq:expansionFSfibre},
\begin{eqnarray*}
   f_k(x) &=& \vol\,\sum_i c_i \frac{h_{FS,k}^{k+i}}{h^{k+i}} \sum_{\alpha} |s_{\alpha}(x)|_{h}^2\\
&=&\vol\,\sum_i c_i ( 1 + \frac{\Scal(\omega)-\overline{\!S}}{2k} + O(k^{-2})) \sum_{\alpha} |s_{\alpha}(x)|_{h}^2 \\
&=&b_0k^n+O(k^{n-1}),
\end{eqnarray*}
by \eqref{another}, where $b_0$ is constant.
\end{proof}

\chapter{Limits of balanced metrics}\label{chap:limitsbalanced}

We digress in this section from our proof of Donaldson's Theorem to give another application of the weighted Bergman kernel that illustrates the connection between balanced metrics and metrics of constant scalar curvature.  

\begin{thm}   Let $(h_k,\omega_k)$ be a pair that is balanced for the embedding $X\subset \PP(\oplus_i H^0(L^{k+i})^*)$, and suppose this sequence converges in $C^2$ to a limit $(h,\omega)$.  Then $2\pi \omega$ is the curvature of $h$ and $\Scal(\omega)$ is constant.
\end{thm}
\begin{proof}
Letting $2\pi\omega_{h_k}$ denote the curvature of $h_k$, by
Lemma \ref{lem:fubinistudyreduction} we have
\begin{equation} \label{eq:reductionmetric3}
\omega_k=\omega_{h_k}+\frac i{2c}\partial\overline\partial f_k,
\end{equation}
where 
\[ f_k(x)=\vol\,\sum_i c_i \sum_{\alpha} |s^i_{\alpha}(x)|^2_{h_{k}}\]
and $\{s^i_{\alpha}\}$ is a graded orthonormal basis of $\oplus_i H^0(L^{k+i})$ with
respect to the $L^2$-metric defined by $(h_k,\omega_k)$. (Here we are using the balanced
condition: that
$(h_k,\omega_k)$ is the Fubini-Study metric induced from this $L^2$-metric.)

By \eqref{another} we have the $C^4$-estimate
\begin{equation}
 f_k =\vol\frac{\omega_{h_k}^n}{\omega_k^n} \sum_i c_i k^n  + O(k^{n-1}) \label{eq:expansionoff}
\end{equation}
(The estimate is in $C^4$ rather than $C^2$ since we only require it to top order.  Moreover we have used here that the sequence $(h_k,\omega_k)$ converges so lies in a compact set, and thus the $O(k^{n-1})$ can be taken uniformly.)     Since $c=\sum_i c_i(k+i) h^0(L^{k+i})$ is of order $\Omega(k^{n+1})$, we deduce from \eqref{eq:reductionmetric3} that $\omega_k = \omega_{h_k} + O(k^{-1})$ in $C^2$. 

In turn this implies that $\omega_{h_k}^n/\omega_k^n = 1 + O(k^{-1})$, which we can feed back into \eqref{eq:expansionoff} to give $\partial
\overline\partial f_k = O(k^{n-1})$.    Hence in fact $$\omega_k = \omega_{h_k} + O(k^{-2}).$$
In particular, taking the limit as $k\to\infty$ implies that $\omega=\omega_h$, i.e. that $2\pi\omega$ is the curvature of $h$. \medskip

Therefore $\omega_{h_k}^n/\omega_k^n = 1+ O(k^{-2})$ and
$$
\tr_{\omega_{h_k}}(\Ric(\omega_k))= \tr_{\omega_k}(\Ric(\omega_k)) + O(k^{-2}) = \Scal(\omega_k) + O(k^{-2}).
$$
Thus the asymptotic expansion \eqref{eq:expansion} for the weighted Bergman kernel becomes
\begin{equation} \label{Bkexp}
B_k=\vol \sum_i (k+i)c_i \sum_{\alpha} |s^i_{\alpha}|^2 =\vol \sum_i c_i k^{n+1} + b_1 k^n + O(k^{n-1}),
\end{equation}
where  $b_1 = \vol \sum_i c_i \left((n+1)i + \frac{1}{2} \Scal(\omega_k)\right)$. 
But by Proposition \ref{prop:balancedimpliesBergmanconstant} the balanced condition implies
that this weighted Bergman kernel is the constant
\[c=\vol \sum_i c_i k^{n+1} + O(k^n).\]
So the coefficient of $k^{n+1}$ agrees with that of \eqref{Bkexp}.
Taking coefficients of $k^n$ gives, after some rearranging, a constant
$\,\overline{\!S}$ independent of $k$  such that
\[ \Scal(\omega_k) - \,\overline{\!S} = O(k^{-1}).\]   Taking $k$ to infinity yields $\Scal(\omega)= \,\overline{\!S}$ as required.
\end{proof}

\begin{rmk}
The previous theorem was first observed by Donaldson \cite{donaldson(01):scalar_curvat_projec_embed} in the case of manifolds embedded in projective space.  In the same paper Donaldson also proves a much harder converse: a cscK metric implies the existence of balanced metrics for large $k$.  We expect that this converse can also be generalised to orbifolds embeddings in weighted projective space, but have not attempted to prove it.
\end{rmk}

\chapter{K-stability as an obstruction to orbifold cscK metrics}\label{chap:KstabDon}

We now have the tools required to prove the orbifold version of Donaldson's Theorem, and start with the precise definition of stability.

\section{Definition of orbifold K-stability}

Fix a compact $n$-dimensional polarised orbifold with cyclic quotient
singularities $(X,L)$.

\begin{defn}
  A test configuration for $(X,L)$ consists of a pair $(\pi\colon \mathcal
  X\to \mathbb C,\mathcal L)$ where $\mathcal X$ is an orbischeme, $\pi$ is flat and $\mathcal L$
  is an ample orbi-line bundle along with a $\C^*$-action such
  that (1) the action is linear and covers the usual action on $\mathbb C$
  and (2) the  general fibre $\pi^{-1}(t)$ of the test configuration is $(X,L)$.
\end{defn}

Test configurations arise from the action of a one parameter $\C^*$-subgroup of the automorphisms
of weighted projective space $\PP$ on an orbifold embedded in $\PP$.
In general the limit $\mathcal X_0=\pi^{-1}(0)$ will not itself be an orbifold, as it may have scheme structure or entire components consisting of points with nontrivial stabilisers.
In general one should allow $\mathcal X$ to be a Deligne-Mumford stack, but for most of the applications in this paper $\mathcal X$ will itself be an orbifold.

Conversely, we can realise an abstract test configuration via a $\C^*$-action on weighted
projective space, just as in the manifold case
\cite[Proposition 3.7]{ross_thomas:07:study_hilber_mumfor_criter_for}.
Using the orbi-ampleness of $\mathcal L$, we can embed $\mathcal X$ into the
weighted projective bundle $\PP(\oplus_i(\pi_*\mathcal L^{k+i})^*)$ over the base curve
$\C$ for $k\gg0$, such that the pullback of $\O_\PP(1)$ is $\mathcal L$. Pick a trivialisation
of the bundle, making it isomorphic to $\PP(V)\times\C$, where $V=\oplus_iH^0(\mathcal
X_0,\mathcal L^{k+i}|_{\mathcal X_0})^*$. Thus the $\C^*$-action on $V$ arising from
the one on the central fibre $(\mathcal X_0,\mathcal L_0)$ induces a diagonal $\C^*$-action
on $\PP(V)\times\C\supset\mathcal X$ giving the original test configuration.
\medskip

By Proposition \ref{prop:leadingtermweight} we can write the total weight of the $\C^*$-action
on $H^0(L^k)$ as
\begin{equation} \label{w}
w(H^0(L^{k})) = w(k) + \tilde o(k^n),
\end{equation}
where $w(k)$ is a polynomial $b_0k^{n+1}+b_1k^n$ of degree $n+1$. Similarly
\begin{equation} \label{h}
h^0(L^k)=h(k)+\tilde o(k^{n-1}),
\end{equation}
where $h(k)=a_0k^n+a_1k^{n-1}$.

\begin{defn}\label{def:futaki}
  The \emph{Futaki invariant} of the test configuration $(\mathcal X,\mathcal L)$ is the $F_1=\frac{a_0b_1-a_1b_0}{a_0^2}$ term in the expansion
  \[\frac{w(k)}{kh(k)} = F_0 + \frac{F_1}{k} + O\left(\frac{1}{k^2}\right) .\]
  We say $(X,L)$ is \emph{K-semistable} if $F_1\ge 0$ for any
  test configuration with general fibre $(X,L)$.  We say it is \emph{K-polystable} if in addition $F_1=0$ only if the test configuration is a product $\mathcal X=X\times \mathbb C$, i.e. it arises from a $\C^*$-action on $X$.
\end{defn}

In other words we are simply ignoring the non-polynomial terms in the Hilbert and weight functions, and then defining stability exactly as for manifolds. 

One reason this is a sensible stability notion related to scalar curvature is given by
our next result. This shows that taking a weighted sum with our choice of $c_i$ kills
the periodic terms, a result we will apply later to both $w$ \eqref{w} and $h$ \eqref{h}.

\begin{lem}\label{lem:ignoringperiodic}
Let $H$ be a function of the form
\[ H(k) = h(k) + \epsilon_h(k),\]
where $h$ is a  polynomial of degree $n$ and $\epsilon_h$ is a sum of terms of the form $r(k)\delta(k)$ where $r$ is a polynomial of degree $n-1$ and $\delta(k)$ is periodic with period $m$ and average zero.  Then
\[ \sum_{i} c_i H(k+i) = \sum_i c_i h(k+i) + O(k^{n-4}).\]
\end{lem}
\begin{proof}
First we claim that if $0\le p\le 3$ then $\sum_{i\equiv u} c_i i^p$ is independent of $u$.    To see this, let $m=\ord(X)$ and observe that by \eqref{eq:choiceofci},  $\sum_i c_i t^i$ has a root of order at least 4 at every non-trivial $m$th root of unity.  Thus if  $\sigma^m=1$ with $\sigma\neq 1$ then $\sum_i c_i i^p \sigma^{ri}=0$ for $1\le r\le m-1$.  So given any $u$,
\[\sum_i i^p c_i = \sum_{r=0}^{m-1} \sigma^{-ru} \sum_i i^p c_i \sigma^{ri}=\sum_i i^p c_i\left(\sum_{r=0}^{m-1}\sigma^{(i-u)r}\right)=m\sumskip{i^p c_i},\]
which proves the claim.

We have to show that $\sum_i c_ir(k+i) \delta(k+i) = O(k^{n-4})$. By
the claim,
\begin{eqnarray*}
  \sum_i c_i  i^p  \delta(k+i) &=& \sum_{u=1}^m \sum_{i\equiv u-k \text{ mod }m } c_i  i^p \delta(u) \\
&=&\frac{1}{m} \sum_{u=1}^m \delta(u) \sum_i c_i i^p\ =0
\end{eqnarray*}
for $0\le p\le 3$. Hence the $k^d,\ldots,k^{d-3}$ terms in $\sum_i c_i (k+i)^d \delta(k+i)$
vanish, and the sum is $O(k^{n-4})$ if $d\le n$.
The result for general polynomials $r$ follows by linearity.
\end{proof}

\section{Orbifold version of Donaldson's theorem}

To recall the general setup, let $h$ be a hermitian metric on $L$ with positive curvature $2\pi\omega$ and for $k\gg 0$ consider the $\Hilb(h,\omega)$ metric on $\oplus_i H^0(L^{k+i})$ from \eqref{eq:definitionofHilb}.  From the embedding $X\subset \PP(\oplus_i H^0(L^{k+i})^*)$ we produced in Definition \ref{def:balanced} a hermitian matrix $M(X) = M_k(X)$ (and defined the embedding to be balanced at level $k$ when $M_k(X)$ vanishes). Using the norm  $\|A\|^2 = \tr(AA^*)$ on hermitian matrices, the following is the key estimate.

\begin{thm}\label{thm:boundonbalancedmatrix}
  There is a constant $C$ such that
$$\lVert M_k(X) \rVert \le Ck^{\frac{n-2}{2}}\lVert \Scal(\omega)-\,\overline{\!S}\rVert_{L^2} + O(k^{\frac{n-4}{2}}),$$
where the $L^2$-norm is taken with respect to the volume form determined by $\omega$.
\end{thm}

\begin{proof}
  To ease notation we write $M= M_k(X)$.  Since $\|M\|$ is unchanged by a unitary transformation we may pick $\Hilb(h,\omega)$-orthonormal coordinates
$\{t_{\alpha}^i\}$ such that $M = \oplus M^i$ with each $M^i$ diagonal.  Thus $M^i$ has entries \eqref{lem:Mincoordinates}
\[ M^i_{\alpha\alpha} = \frac{1}{2}\left( \int_X |t^i_\alpha|_{h_{FS,k}}^2 \frac{\omega^n_{FS,k}}{n!}-c_i\vol\right)=\frac{1}{2}\left(\int_X |t^i_{\alpha}|_h^2 \frac{h_{FS,k}^{k+i}}{h^{k+i}} \frac{\omega_{FS,k}^n}{n!}-c_i\vol\right)\!,\]
where $h_{FS,k}$ and $\omega_{FS,k}$ are the induced Fubini-Study metrics.  Using the expansion of $h_{FS,k}/h = 1 + (\,\overline{\!S}-S)/2 k^{2}+O(k^{-3})$ of Theorem
\ref{thm:expansionFS} we can write $M = A+B$ where $B^i_{\alpha\alpha}= O(k^{-2})$ and
\begin{align*}
 A^i_{\alpha\alpha} &=\frac{1}{2}\left(\int_X |t^i_{\alpha}|_h^2 \left(1 + \frac{\,\overline{\!S}-S}{2k}\right)\frac{\omega^n}{n!} -c_i\vol\right)= \frac{1}{4k}\int_X |t^i_{\alpha}|_h^2 (\,\overline{\!S}-S)\frac{\omega^n}{n!}\,.
\end{align*}
Here we have used $\|t^i_{\alpha}\|^2_{L^2}=c_i\vol$ from the definition of the
$\Hilb(h,\omega)$-norm.   Using the Cauchy-Schwarz inequality,
\begin{eqnarray*}
  | A^i_{\alpha\alpha}|^2&\le& \frac{1}{16k^2}\int_X |t^i_{\alpha}|_h^2 \frac{\omega^n}{n!} \int_X |t^i_{\alpha}|_h^2  |\,\overline{\!S}-S|^2 \frac{\omega^n}{n!} \\
&\le& \frac{C'}{k^2} \int_X |t^i_{\alpha}|_h^2  |\,\overline{\!S}-S|^2 \frac{\omega^n}{n!},
\end{eqnarray*}
for some constant $C'$.  Thus from the weak form of the expansion $\sum_i \sum_{\alpha} |t^i_{\alpha}|_h^2 = O(k^{n})$,
\begin{eqnarray*}
 \|A\|^2 &\le& \frac{C'}{k^2}\int_X \sum_{i,\alpha}  |t^i_{\alpha}|_h^2  |\,\overline{\!S}-S|^2 \frac{\omega^n}{n!}\\
&\le & C''k^{n-2} \|S-\,\overline{\!S}\|_{L^2}^2.
\end{eqnarray*}

Therefore $\|M\| \le \|A\| + \|B\| \le Ck^{\frac{n-2}2}\|S-\,\overline{\!S}\|_{L^2}
+\|B\|$, where $C=\sqrt{C''}$. But $B$ is diagonal with $O(k^n)$ entries of size $O(k^{-2})$, so $\|B\|^2 = O(k^{n-4})$.
\end{proof}

Now let $(\mathcal X,\mathcal L)$ be a nontrivial test configuration for $(X,L)$, embedded
in $\PP(V)\times \mathbb C$ (with $V=\oplus_i H^0(X_0,L_0^{k+i})^*$) as before, induced by a $\mathbb C^*$-action on $\PP(V)$ that takes $X$ to the limit $X_0$. Suppose that
$L$ has a metric $h$ with positive curvature $2\pi\omega$, inducing the Hilb$(h,w)$-metric
on $\oplus_iH^0(X,L^{k+i})$. Applying \cite[Lemma 2]{donaldson:05:lower_bound_calab_funct} to each of the spaces $H^0(X,L^{k+i})$, we get a metric on $V$ such that the induced $S^1$-action is unitary. Therefore the infinitesimal generator $A^{k+i}$ of the induced action on $H^0(L_0^{k+i})$ is hermitian. We set
\begin{equation} \label{defA}
A:=\bigoplus_iA^{k+i}.
\end{equation}
% Pick orthonormal coordinates $z_{\alpha}^i$ for $V$, with respect to which $A^{k+i}$
% has components $A^{k+i}_{\alpha\beta}$. 
As in Section \ref{ERR}, we get a hamiltonian $H_A$ for the $S^1$-action on $\PP(V)$ by contracting its vector field on the circle bundle
$S(\O_{\PP(V)}(1))$ with the connection 1-form whose curvature is $\omega_{FS}$. It is
\[ H_A\colon\PP(V)\to\R, \qquad H_A([v])=\frac1c\sum_i \lambda^{2(k+i)}(v)\langle
A^{k+i}v_{k+i},v_{k+i}\rangle,\]
using the inner product $\langle\cdot,\cdot\rangle$ on $V$ and $\lambda$ as defined in \eqref{eq:defoflambda}. This differs from our usual hamiltonian $m_A=\tr(mA)$ of
\eqref{eq:momementmapweightedprojcspace} by the additive constant $\sum_ic_i\dim V^{k+i}$,
and by the multiplicative factor $\frac1c$. (The latter scaling compensates for the fact
that $m_A$ is the hamiltonian for $c\omega_{FS}$; see Definition \ref{def:FS}.)

By Proposition \ref{prop:leadingtermweight}, then, the polynomial part of the total weight
of the $\C^*$-action on $V^*$ is $w(k) = b_0k^{n+1} + b_1k^n$, where $b_0=\int_{X_0}
H_A \frac{\omega_{FS}^n}{n!}$. From this we can define the Futaki invariant $F_1(\mathcal X,\mathcal L)$ of
the test configuration $(\mathcal X,\mathcal L)$ as in Definition \ref{def:futaki}.

\begin{thm}
In the set-up as above, suppose that $\omega$ has constant scalar curvature. Then $F_1
(\mathcal X,\mathcal L)\ge0$.
\end{thm}

\begin{proof}
In the notation above, set $s=\log t$ and let $X_t=\exp(sA).X$ denote the fibre of the given test configuration over $t\in\C$, with central fibre $X_0$ the limit of $\exp(sA).X$
as $s\to-\infty$.

For fixed $k$, $\tr(M_k(X_s)A)$ is an increasing function of $s\in\R$,
because $\tr(M(X)A)$ is a hamiltonian for the action of $\exp(sA)$ on the space of
sub-orbifolds of $\PP(V)$. Explicitly, substituting $v=Jv_A$ into
\eqref{mmcalc} shows that the derivative of $\tr(M_k(X_s)A)$ is $\Omega(Jv_A,v_A)>0$.  Therefore
$$
\tr(AM_k(X))\ =\ \tr(AM_k(X_1))\ \ge\ \lim_{s\to-\infty}
\tr(AM_k(X_t))\ =\ \tr(AM_k(X_0)).
$$
Recalling the definition of $M_k(X)$ \eqref{def:balanced}, this gives
\begin{eqnarray} \nonumber
\|A\|\|M_k(X)\|\ \ge\ \int_{X_0}m_A\frac{\omega^n_{FS}}{n!} &=& 
\int_{X_0}cH_A\frac{\omega^n_{FS}}{n!}-\vol\sum_i c_i \tr(A^{k+i}) \\
&=& cb_0-a_0\sum_i c_iw(H^0(L^{k+i})),\label{eq:momentmapinequality}
\end{eqnarray}
by Proposition \ref{prop:leadingtermweight}. Here we are writing 
$h^0(L^k) = a_0 k^n + a_1 k^{n-1} + \widetilde o(k^{n-1})$ and 
$w(H^0(L^k))=b_0k^{n+1}+b_1k^n+\widetilde o(k^n)$. Lemma \ref{lem:ignoringperiodic}
then gives
\begin{eqnarray*}
c=\sum_i c_i (k+i)h^0(L^{k+i}) &=& \tilde{a}_0k^{n+1}+\tilde{a}_1 k^n+O(k^{n-1}),\\
\text{and}\quad \sum_i c_iw(H^0(L^{k+i})) &=& \tilde{b}_0 k^{n+1} + \tilde{b}_1 k^n + O(k^{n-1}), 
\end{eqnarray*}
where
\begin{eqnarray*}
\tilde{a}_0=a_0 \sum_i c_i  \ &\text{and}& \ \tilde{a}_1=\sum_i c_i (a_0i(n+1) + a_1), \\
\tilde{b}_0=b_0 \sum_i c_i  \ &\text{and}& \ \tilde{b}_1=\sum_i c_i (b_0i(n+1) + b_1).
\end{eqnarray*}

% \begin{align*}
% c=  \sum_i c_i (k+i) h^0(L^{k+i}) &= \tilde{a}_0 k^{n+1} + \tilde{a}_1 k^n + O(k^{n-1}),
% \end{align*}
% where
% \begin{align*}
% \tilde{a}_0=a_0 \sum_i c_i  \quad \text{and} \quad \tilde{a}_1=\sum_i c_i (a_0i(n+1) + a_1).
% \end{align*}
% Similarly
% \begin{align*}
%   \sum_i c_iw(H^0(L^{k+i})) &= \tilde{b}_0 k^{n+1} + \tilde{b}_1 k^n + O(k^{n-1}), 
% \end{align*}
% where
% \begin{align*}
% \tilde{b}_0=b_0 \sum_i c_i  \quad \text{and} \quad \tilde{b}_1=\sum_i c_i (b_0i(n+1)
% + b_1).
% \end{align*}
Therefore \eqref{eq:momentmapinequality} becomes
\begin{eqnarray*}
  \norm{A}\norm{M_k(X)} &\ge& c \left(b_0- a_0\frac{\tilde{b}_0k^{n+1} + \tilde{b}_1 k^n + O(k^{n-1})}{\tilde{a}_0k^{n+1} + \tilde{a}_1 k^n + O(k^{n-1})}\right)\\
&=&  ca_0 \left(k^{-1}\frac{\tilde{b}_0\tilde{a}_1-\tilde{b}_1\tilde{a}_0}{\tilde{a}_0^2}+ O(k^{-2})\right)\\
&=&  ca_0\left(k^{-1}\frac{b_0a_1-b_1a_0}{a_0^2}+ O(k^{-2})\right) \\
&=&  ca_0\left(-k^{-1}F_1+ O(k^{-2})\right).
\end{eqnarray*}
Now $\|A\|^2=|\tr A^2|=O(k^{n+2})$ by \eqref{squaredweight}, and $c$
is strictly of order $O(k^{n+1})$. So Theorem
\ref{thm:boundonbalancedmatrix} now gives
$$
k^{\frac{n-2}{2}}\lVert \Scal(\omega)-\,\overline{\!S}\rVert_{L^2} + O(k^{\frac{n-4}{2}})
\ge Ck^{\frac n2}\left(-k^{-1}F_1+ O(k^{-2})\right)
$$
for some constant $C>0$. Hence when $\Scal(\omega)$ is constant (and therefore equal to
$\,\overline{\!S}$) we see that $F_1\ge 0$.
\end{proof}

\begin{cor} \label{cor}
Let $(X,L)$ be a polarised orbifold with cyclic stabiliser groups. If $X$ admits an orbifold K\"ahler metric $\omega\in\mathcal K(c_1(L))$ with constant scalar curvature then $(X,L)$
is K-semistable.
\end{cor}

\chapter{Slope stability of orbifolds}\label{sec:slope}

To get examples where K-stability obstructs the existence of constant scalar curvature metrics we need a supply of test configurations for which we can calculate the Futaki invariant.  To this end we briefly describe the notion of slope stability. The detailed
descriptions in \cite{ross_thomas:06:obstr_to_exist_const_scalar,ross_thomas:07:study_hilber_mumfor_criter_for} extend easily from manifolds to orbifolds
with a few minor changes.
\medskip

Fix an $n$-dimensional polarised orbifold $(X,L)$ and a sub-orbischeme (or substack)
$Z\subset X$: an invariant subscheme $Z_U$ in each orbifold chart $U\to U/G\subset X$
such that for
each injection of charts $U'\into U$, the subscheme $Z_{U'}$ is the scheme-theoretic
intersection $Z_U\cap U'$. In most of our examples $Z$ will be smooth but with generic
stabilisers. Working equivariantly in charts one can produce a new orbischeme,
the blowup $\pi\colon \hat{X}\to X$ of $X$ along $Z$. Locally this is the blow up of
$U$ in $Z_U$ divided by the induced action of the Galois group on this blow up.
The exceptional divisors glue to give an orbifold exceptional divisor $E\subset \hat{X}$.

For large $N$, $\pi^*L^N(-E)$ is positive. (From now on we will suppress $\pi^*$.)
Thus we can define the \emph{Seshadri constant}
by
\[ \epsilon_{\orb}(Z)= \text{sup}\left\{x\in \mathbb Q_+ \colon (L(-xE))^M \text{ is ample for some } M\in\mathbb N\right\}.\]
For example, if we put $\ZZ/m\ZZ$ stabilisers along a smooth divisor $D\subset X$
then as in Section \ref{Qdiv} there is a well defined orbi-divisor $D/m$ whose Seshadri constant $\epsilon_{\orb}(D/m)
=m\epsilon(D)$ is $m$-times the usual Seshadri constant of $D$ in the underlying space of $X$.
% Moreover if $c=a/b<\epsilon_{\orb}(Z)$ is sufficiently general then $L^a(-bE)$ will
% be a polarisation on the orbifold
% blow up -- i.e. it will also be locally ample.

To get a test configuration from $Z$,  consider the suborbifold $Z\times \{0\}\subset X\times \mathbb C$.  Blowing this up gives the degeneration $\mathcal X\to X\times \mathbb C\to \mathbb C$ to the normal cone of $Z$ with exceptional divisor $P$. As shown in
\cite[Proposition 4.1]{ross_thomas:07:study_hilber_mumfor_criter_for} for schemes (and
the same results go through easily for orbifolds), $\epsilon_{\orb}(Z\times\{0\})=
\epsilon_{\orb}(Z)$.  Let $p\colon \mathcal X\to X$ be the projection.  Then for generic
$c\in(0,\epsilon_{\orb}(Z))\cap\Q$ general integer powers of $\mathcal L_c:=p^*L(-cP)$ define a polarisation of $\mathcal X$.  The natural action of $\mathbb C^*$ on $X\times \mathbb C$ (trivial on $(X,L)$, weight one on $\C$) lifts naturally to a linearised action on  $(\mathcal X,\mathcal L)$, and thus for such $c$ we have a test configuration $(\mathcal X,\mathcal L_c)$ with general fibre $(X,L)$.  The central fibre is $\mathcal X_0 = \hat{X}\cup_E P$ consisting of the blowup $\hat{X}\to X$ along $Z$ glued to $P$ along $E$, and the induced $\mathbb C^*$ action is trivial on $\hat{X}$ and acts by scaling $P$ along the normal to $E$. 

As usual we write
\begin{equation}\label{eq:hilbertpoly}
h^0(L^k) = a_0k^n + a_1k^{n-1} + \tilde o(k^{n-1}),
\end{equation}
and then define the \emph{slope} of $(X,L)$ to be
\[ \mu(X,L) = \frac{a_1}{a_0}=-\frac{n\int_X c_1(K_{\orb}).c_1(L)^{n-1}}{2\int_X c_1(L)^n}\,,\]
by \eqref{orbRR}.
To define the slope of $Z\subset X$, we work on the orbifold blowup $\pi\colon \hat{X}\to X$ along $Z$ with exceptional divisor $E$. Then orbifold Riemann-Roch to $L^k(-\frac
jkE)$ for fixed $j$ (and $k=jK$ for some integer $K$) takes the form
\begin{equation}
h^0(L^k(-jE)) = p(k,j) + \epsilon_p(k,j). \label{eq:weightpoly}
\end{equation}
Here $p$ is a polynomial of two variables of total degree $n$ and $\epsilon_p$ is a sum of terms of the form $r_p\delta'$ where $r_p$ is a polynomial of two variables of total degree $n-1$ and $\delta'=\delta'(k,j)$ is periodic in each variable with average
$\sum_{k,j=1}^M \delta'(k,j)=0$.  Define polynomials $a_i(x)$ by
\[p(k,xk)= a_0(x) k^n + a_1(x) k^{n-1} + O(k^{n-2}) \quad \text{for } kx\in \mathbb N.\]
Then the \emph{slope} of $Z$ (with respect to $c$) is 
\begin{equation}
\mu_c(\I_Z) : = \frac{\int_0^c a_1(x) + \frac{a'_0(x)}{2} dx}{\int_0^c a_0(x) dx}\,.
\label{eq:defofslopeIz}
\end{equation}
The only difference from the manifold case is that we ignored the periodic terms
in the relevant Hilbert functions. This amounts to replacing $K_X$ by $K_{\orb}$.

\begin{defn}
  We say that $(X,L)$ is \emph{slope semistable with respect to $Z$} if 
\[ \mu_c(\I_Z) \le \mu(X) \quad \text{ for all } 0<c<\epsilon_{\orb}(Z).\]
We say that $X$ is \emph{slope semistable} if it is slope semistable with respect to
all sub-orbischemes $Z\subset X$.
\end{defn}

Alternatively, just as in the manifold case \cite[Definition 3.13]{ross_thomas:06:obstr_to_exist_const_scalar}
we can put $\tilde{a}_i(x): = a_i-a_i(x)$
and define the \emph{quotient slope} of $Z$ as
\begin{equation}
\mu_c(\O_Z) : = \frac{\int_0^c \tilde{a}_1(x) + \frac{\tilde{a}'_0(x)}{2} dx}{\int_0^c \tilde{a}_0(x) dx}\,,\label{eq:defofslopez}
\end{equation}
and $(X,L)$ is slope semistable with respect to $Z$ if and only if $\mu(X)\le\mu_c(\O_Z)$
for all $0<c<\epsilon_{\orb}(Z)$.

One can check easily that slope semistability is invariant upon replacing $L$ by a positive power.  The point of these definitions is that the sign of the Futaki invariant of the test configuration given by deformation to the normal cone of $Z$ is the same as the sign of $\mu(X)-\mu_c(\I_Z)$, resulting in the following slope obstruction to stability.

\begin{thm} \label{slopeK}
  If $(X,L)$ is K-semistable then it is slope semistable.
\end{thm}

\begin{proof}
  The argument is essentially the same as that in the smooth case
  \cite[Section 4]{ross_thomas:07:study_hilber_mumfor_criter_for}; only the Riemann-Roch
  formula changes. Since being not slope semi\-stable is an open condition, we may without loss of generality assume that $c<\epsilon_{\orb}(Z)$ is general, and so by rescaling $L$ be may assume that $c$ is integral and coprime to $m$, making $(\mathcal X,\mathcal L_c)$ a test configuration.   The space of sections on the central fibre of the test configuration splits as
\renewcommand{\I}{\mathcal I}
\begin{equation} \label{eq:eigenspace}
H^0_{\mathcal X_0}(\L_c^k)=H^0_X(L^k\otimes\I_Z^{ck})\ \oplus\,\bigoplus_{j=1}^{ck}t^j \frac{H^0_X(L^k\otimes\I_{Z}^{ck-j})}{H^0_X(L^k\otimes\I_Z^{ck-j+1})}\,.
\end{equation}
Here $H^0_X(L^k\otimes\I_Z^j)$ is the space of sections of $L^k$ which vanish to order
$j$ on $Z$ (in an orbi chart). The coordinate $t$ is pulled back from the base $\C$,
and is acted on by $\C^*$ with weight $-1$. Therefore \eqref{eq:eigenspace} is the weight
space decomposition of $H^0_{\mathcal X_0}(\L_c^k)$, with total weight 
\begin{eqnarray*}
w(H^0_{\mathcal X_0}(\L_c^k))  &=& -\sum_{j=1}^{ck} j\Big(h^0_X(L^k\otimes\I_{Z}^{ck-j})-
h^0_X(L^k\otimes\I_Z^{ck-j+1})\Big).
\end{eqnarray*}
Some manipulation, and the vanishing of higher cohomology of the pushdowns of these sheaves
to the underlying scheme, give
\begin{eqnarray*}
\sum_{j=1}^{ck} h_X^0(L^k\otimes \I_Z^j)-ckh_X^0(L^k) &=& \sum_{j=1}^{ck} h^0_{\hat{X}}
(L^k(-jE))-ckh^0(L^k) \\
&=& \sum_{j=1}^{ck} p(k,j) + \epsilon_p(k,j) - ck  [h(k) + \tilde o(k^{n-1})].
\end{eqnarray*}
By Lemma \ref{lem:summingperiodic} below, the periodic terms to not contribute to the top two order parts of this sum, so the leading order polynomial parts of the weight
are
\[ w(k) = \sum_{j=1}^{ck} p(j,k) - ckh(k) + \tilde o(k^n). \]
The calculation of the Futaki invariant is now exactly as in the smooth case
\cite[Proposition 4.14 and Equation 4.19]{ross_thomas:07:study_hilber_mumfor_criter_for},
yielding
$$
F_1(\mathcal X,\mathcal L_c)=(\mu(X)-\mu_c(\I_Z))\frac{\int_0^ca_0(x)dx}{a_0}\,.
$$
This is nonnegative if and only if $X$ is slope semistable with respect to $Z$.
\end{proof}

\begin{lem}\label{lem:summingperiodic}
  Suppose $\delta(k,j)$ is periodic in each variable with period $m$ and average $\sum_{k,j=1}^m \delta(k,j)=0$.  Suppose also that $r(k,j)$ is a polynomial of two variables of total degree $n-1$ and $c$ is a fixed integer.  Then
\[ \sum_{j=1}^{ck} r(k,j) \delta(k,j) =\epsilon(k) + O(k^{n-1}),\]
where $\epsilon(k)$ is a sum of terms of the form $r(k)\delta'(k)$, with $r$ a polynomial of degree $n$ and $\delta'$ periodic of average zero.
\end{lem}
\begin{proof}
  By linearity it is sufficient to consider the case where $r(k,j)=j^{n-1}$.    Set $j_0=\Lfloor \frac{ck}{m}\Rfloor m$ and split the sum into two pieces depending on whether $j\le j_0$ or $j\ge j_0+1$.    In the first case, writing $j=mu+v$,
\[\sum_{j=0}^{j_0} j^{n-1}\delta(k,j) = \sum_{v=1}^m \delta(k,v) \sum_{u} (um+v)^{n-1}.\]
This splits into pieces of the form $P(k)\epsilon(k)$ where $\epsilon(k)$ is periodic and $P$ is a polynomial of degree at most $n$, with the degree being equal to $n$ only when $\epsilon(k) = \sum_{v=1}^m \delta(v,k)$ in which case $\epsilon$ has average zero.  Then note that if $P$ has degree $n-1$ there is a constant $a$ so that $\epsilon-a$ has average zero, and $P\epsilon = P(\epsilon-a) + aP = P(\epsilon-a) + O(k^{n-1})$.  Thus, after some rearrangement, this part of the sum is of the form claimed.  The sum for $j\ge j_0+1$ immediately splits into terms of the form $P(k)\epsilon(k)$ where $P$ has degree at most $n-1$ and $\epsilon$ is periodic, so by the same argument these terms are also of the required form.
\end{proof}

We can calculate the slope of (sufficiently nice) suborbifolds much as in the manifold
case. For instance let $\hat{X}\stackrel{\pi}{\to} X$ be the orbifold blowup along a smooth $Z$ of codimension $r\ge 2$, with orbifold exceptional divisor $E$. Then
$K_{\orb,\hat{X}} = \pi^* K_{\orb,X} + (r-1)E$, and
\begin{eqnarray}\label{eq:airroch}
  a_0(x) &=& -\frac{\int_{\hat{X}}c_1(L(-xE))^n}{n!}\,,\nonumber\\
  a_1(x) &=& -\frac{\int_{\hat{X}} c_1(K_{\orb,\hat{X}})c_1(L(-xE))^{n-1}}{2(n-1)!}\,.
\end{eqnarray}
So the formulae only differ from those in \cite{ross_thomas:06:obstr_to_exist_const_scalar} in replacing $K_X$ by $K_{\orb}$. For example if $Z$ is as small as possible -- the invariant
subvariety defined by a reduced fixed point upstairs in an orbifold chart -- then these
quickly imply
\begin{equation}
\mu_c(\O_Z) = \frac{n(n+1)}{2c}\,.\label{eq:slopeofapoint}
\end{equation}
Notice the order of the stabiliser group at this point does not feature; however it enters
into the Seshadri constant of $Z$ and so does affect slope stability.

Similarly if $Z$ is an orbifold divisor in an orbifold surface then
$$
\mu_c(\O_Z) = \frac{3(2L.Z - c(K_{\orb}.Z + Z^2)}{2c(3L.Z-cZ^2)}\,.
$$

\chapter{Applications and further examples}\label{chap:examples}

\section{Orbifold Riemann surfaces}\label{sec:orbifoldriemannsurfaces}

By an orbifold Riemann surface we mean an orbifold of complex dimension one. This is equivalent to the data of a Riemann surface of genus $g\ge 0$ and $r$ points
$p_1,\dots,p_r\in X$ marked by orders of stabiliser groups $m_1,\ldots,m_r\ge 2$. We
assume $r\ge1$.

\begin{thm}\label{thm:slope_orbifold_riemann_surface}
A polarised Riemann surface $(X,L)$ is slope semistable if and only if
  \begin{equation}
 2g + \sum_{i=1}^r \left(\frac{m_i-1}{m_i}\right)\ge  2\max_{i=1,\ldots,r}\left\{\frac{m_i-1}{m_i}\right\}\label{eq:slope_orbifold_riemann_surfacesemi}
\end{equation}
%Moreover $(X,L)$ is slope stable if and only if strict inequality holds. 
\end{thm}

\begin{proof}
   The orbifold canonical bundle of $X$ is
$K_{\orb} = K_X + \sum_{i=1}^r \left(1-\frac{1}{m_i}\right)p_i$
so the slope of $X$ is
\begin{equation} \label{orbslopeX}
\mu(X,L) = -  \frac{\deg K_{\orb}}{2 \deg L} =  \frac{1 -g- \frac{1}{2}\sum_{i=1}^r \left(1-\frac1{m_i}\right)}{\deg L}\,.
\end{equation}
Without loss of generality assume $m_1\ge m_2\ge \dots \ge m_r$.   Let $Z=\{p_1/m\}$
be the orbifold point of order $m_1$ with a reduced lift upstairs.  Then
$$\mu_c(\O_Z) = c^{-1} \quad\text{and}\quad \epsilon_{\orb}(Z,X,L) = m_1 \deg L.$$
Now if $(X,L)$ is semistable then  $\mu_{\epsilon}(\O_Z)\ge \mu(X),$  so
$$ \frac{1}{m_1}\ge 1-g - \frac{1}{2}  \sum_{i=1}^r \frac{m_i-1}{m_i}$$
which rearranges to give the  inequality \eqref{eq:slope_orbifold_riemann_surfacesemi}.

For the converse suppose that \eqref{eq:slope_orbifold_riemann_surfacesemi} holds and consider an orbifold subspace $Z\subset X$.  This $Z$ is an orbifold divisor whose degree is a rational number $q\ge \frac{1}{m_1}$.  Thus $\epsilon:=\epsilon_{\orb}(Z,L)=\frac1q
\deg L \le m_1 \deg L$.   As $\mu_c(\O_Z)=c^{-1}$ is decreasing with respect to $c$, we get $ \mu_c (\O_Z) \ge \epsilon^{-1}\ge(m_1\deg L)^{-1}$. But \eqref{eq:slope_orbifold_riemann_surfacesemi} implies that this is greater than or equal to $\mu(X,L)$ \eqref{orbslopeX}, so $X$ is slope semistable.
%The case of stability is the same only with non-strict inequalities.
\end{proof}

\begin{rmk}\label{rmk:unstablep1}
  \begin{enumerate}
  \item It is hard to either violate or achieve equality in the inequality \eqref{eq:slope_orbifold_riemann_surfacesemi}.
  Either would imply that $2g+0\le2\max(1-\frac1{m_i})<2$
  and so $g=0$. Then since each integer $m_i\ge 2$, we find there are only three cases
  in which an orbifold Riemann surface is not strictly slope stable:
\begin{enumerate}
\item $g=0,\ r=1$ (this gives $\PP(1,m)$),
\item $g=0,\ r=2,\ m_1\neq m_2$ (giving $\PP(m_1,m_2)$ if $\hcf(m_1,m_2)=1$),
and
\item $g=0,\ r=2,\ m_1=m_2$.
\end{enumerate}
In the first two cases $(X,L)$ is not slope semistable and so not cscK.  In the third case $(X,L)$ is actually slope polystable, as we now describe. The only way in which $\mu_c(\O_Z)=\mu(X)$ can occur in the proof of \eqref{thm:slope_orbifold_riemann_surface} is if  $Z=\{p_1/m_1\}$
or $Z=\{p_2/m_2\}$ and $c=\epsilon(Z)=2$. In this case deformation to the normal cone $\mathcal X,\mathcal L_c)$ has $\mathcal L_c$ only semi-ample, pulled back from the contraction
of the proper transform of the central fibre $X\times\{0\}$. This contraction is in fact
a product configuration $X\times\C$ (with a nontrivial $\C^*$-action).
\item The stability condition \eqref{eq:slope_orbifold_riemann_surfacesemi} is actually a special case of \eqref{ex:orbiprojectivespace} and thus an manifestation of the index obstruction which we discuss below.
% \item The previous proof shows that the ``worst'' possible destabilising subspace of an orbifold Riemann surface with orbifold points $p_1,\ldots,p_r$ is given by $Z=\{p_1\}$ where $p_1$ has the highest orbifold order, and the lift of $Z$ is a single reduced point. \footnote{Presumably this is general:  we may take $Z$ to be connected, reduced and the lift to be as ``small'' as possible upstairs?}
  \end{enumerate}
\end{rmk}

Slope stability of orbifold Riemann surfaces fits perfectly into the known theory of orbifold cscK metrics which has been studied by several authors including  Picard \cite{picard:05:de_egrat_de_equat_deltau}, McOwen \cite{mcowen:88:point_singul_confor_metric_rieman_surfac} and Troyanov \cite{troyanov:89:metric_const_curvat_a_spher,troyanov:91:presc_curvat_compac_surfac_with_conic_singul}.   In the terminology of this paper Troyanov's results can be paraphrased as follows:

\begin{cor}[Troyanov]
Let $(X,L)$ be a polarised orbifold Riemann surface.  Then $c_1(L)$ admits an orbifold cscK metric if and only if it is slope polystable.
\end{cor}
\begin{proof}
The main theorems in \cite{troyanov:91:presc_curvat_compac_surfac_with_conic_singul} imply that $X$ admits a cscK metric when strict inequality holds in
\eqref{eq:slope_orbifold_riemann_surfacesemi} i.e. as long as $(X,L)$ is slope stable. (To compare our notation with Troyanov's set $\theta_i =\frac{2\pi}{m_i}$ and $\chi_{\orb}
=-\deg K_{\orb}$. Then this is Theorem A in \cite{troyanov:91:presc_curvat_compac_surfac_with_conic_singul} when $\chi_{\orb}<0$, Proposition 2 when $\chi_{\orb}=0$ and Theorem C when $\chi_{\orb}>0$.) The only way that $(X,L)$ can be slope polystable and not slope stable is case c): if $g=0,r=2$ and $m_1=m_2$. In this case $X$ is the global quotient $\PP^1/(\ZZ/m\ZZ)$ with
orbifold cscK metric descended from the Fubini-Study metric on $\PP^1$.

For the converse,  if $(X,L)$ is not slope polystable then $g=0$ and either a) $r=1$ or b) $r=2$ and $m_1\neq m_2$.  In these two cases $(X,L)$ is not slope semistable, which
by Corollary \ref{cor} and Theorem \ref{slopeK} implies that $X$ does not admit an orbifold cscK metric.
\end{proof}

\begin{rmk}
The statement that if $g=0$ and a) $r=1$ or b) $r=2$ and $m_1\neq m_2$ then $X$ does
not admit a cscK metric has also been proved by Troyanov \cite[Theorem I]{troyanov:89:metric_const_curvat_a_spher}.
Troyanov's work applies much more generally to cone angles not necessarily of the form $2\pi/m$, and
this is also studied further in \cite{chen:98:extrem_hermit_metric_rieman_surfac,chen_li:95:what_kinds_singul_surfac_can,luo92:_liouv_equat_and_spher_convex_polyt}.  We hope to return to cone angles in $2\pi\Q_+$ using the method described in the Introduction.
\end{rmk}

\section{Index obstruction to stability}\label{sec:index}

Recall that the index $\ind(X)$ of a Fano manifold $X$ is defined to be the largest integer $r$ such that $K^{-1}_{X}$ is linearly equivalent to $r D$ for some Cartier divisor $D\subset X$, and it is well known that if $X$ is smooth then $\ind(X) \le n+1$ with equality if and only if $X=\PP^n$.   By contrast for an Fano orbifold $(X,\Delta)$ it is possible that $K^{-1}_{\orb}\cong\O(rD)$ where $D$ is an orbi divisor and $r\in\mathbb N$ is larger than $n+1$. We will show that this prevents $(X,\Delta)$ from being K-stable. In fact
the same is true under the weaker condition that $K_{\orb}^k\cong\O(krD)$ for some $k\in\mathbb N$ and $n+1\le r\in\mathbb Q$.

\begin{thm}(Index Obstruction)
%Any Fano orbifold $(X,K^{-1}_{\orb})$ of index $I(X)>n+1$ is not slope stable, and in particular
%does not admit a K\"ahler-Einstein metric.
Let $(X,K^{-1}_{\orb})\supset D$ be a Fano orbifold and an orbi divisor. Suppose that $K^{-k}_{\orb}\cong\O(kr D)$ for some $k>0$ and $r\in\Q_+$. If $r >n+1$  then $(X,K^{-1}_{\orb})$ is slope unstable, and thus does not admit an orbifold K\"ahler-Einstein metric.
\end{thm}
\begin{proof}
Set $L=K^{-1}_{\orb}$.  Using \eqref{eq:airroch} to calculate the slope,
\begin{eqnarray*}
a_0(x) &=& \frac{1}{n!}\int_X (c_1(L)-xc_1(D))^n = \frac{1}{n!}(r-x)^n \int_X c_1(D)^n,\\
a_1(x) + \frac{ a_0'(x)}{2} &=&  -\frac{1}{2(n-1)!} (r-1)(r-x)^{n-1} \int_X c_1(D)^n.
\end{eqnarray*}
Now $a_0 = \frac{r^n}{n!} \int_X c_1(D)^n$ and $a_1 =
\frac{r^n}{2(n-1)!} \int_X c_1(D)^n$, so $\mu(X,L) = \frac{n}{2}$.
The Seshadri constant of $D$ is $r$.  Using the definition of the slope
\eqref{eq:defofslopez},
\[\mu_r(\O_D) = \frac{(n+1)((n-1)r+1)}{2nr}\,,\]
which is less that $\mu(X)=n/2$ if and only if $r>n+1$.
\end{proof}

\begin{rmk}
At the level of K\"ahler-Einstein metrics the analogous result has already
been proved by Gauntlett-Martelli-Sparks-Yau using the ``Lichnerowicz obstruction" to the existence of Sasaki-Einstein metrics with non-regular Reeb vector fields
\cite[Section 2.2]{jerome:obstr_to_exist_sasak_einst_metric}.  In fact it was their work
that
originally motivated this paper.  In the same paper the authors discuss the ``Bishop
obstruction" which we have been unable to interpret in terms of stability.
\end{rmk}

\begin{example}(Weighted projective space) Consider weighted projective space
$\WPP=\PP(\lambda_0,\ldots,\lambda_n)$, with $\lambda_0\le \lambda_1\le \ldots\le \lambda_n$
\emph{not all equal}. Then $\{x_0=0\}$ defines an effective divisor in $\O(\lambda_0)$, while $K^{-1}_{\orb}\cong\O(\sum_i \lambda_i)$. Since $\sum_i \lambda_i>(n+1)\lambda_0$ the index obstruction shows that
$\WPP$ is unstable, recovering the well known fact that it does not admit an orbifold
cscK metric.
\end{example}

\begin{example}(Orbifold projective space)\label{ex:orbiprojectivespace}\ 
Let $X=\PP^n$ and take $n+2$ hyperplanes $H_1,\ldots,H_{n+2}$ in general
position, and integers $m_i\ge 2$. Setting
$$\Delta = \sum_{i=1}^{n+2} \left(1-\frac{1}{m_i}\right) H_i,$$
we consider the orbifold $(\PP^n,\Delta)$.
Then $K^{-1}_{\orb}=K^{-1}_{\PP^n}(-\Delta)$ becomes equivalent after passing to powers
to
\begin{equation} \label{bigbrackets}
\O\bigg(\!n+1  - \sum_{i=0}^{n+2} \Big(1- \frac{1}{m_i}\Big)\bigg) =\O\bigg(\!\!-1 +
\sum_{i=1}^{n+2} \frac{1}{m_i}\bigg).
\end{equation}
Thus $(\PP^n,\Delta)$ is a Fano orbifold as long as $\sum_{i=1}^{n+2} \frac{1}{m_i}>1$.

The right hand side of \eqref{bigbrackets} can be written
$$
\O\bigg(m_j\bigg(\!\!-1 +
\sum_{i=1}^{n+2} \frac{1}{m_i}\bigg)D_j\bigg),
$$
where $D_j=\frac{1}{m_j}H_j$ is an orbi divisor.
Thus by the index obstruction, if $(X,\Delta)$ is a semistable Fano orbifold then
\begin{equation}
 \sum_{i=1}^{n+2} \frac{1}{m_i} \le  1 + (n+1) \min_{1\le i\le n+2}\left(\frac{1}{m_i}\right).
 \label{eq:stabilityofpnhyperplane}
\end{equation}
\end{example}

\begin{rmk}
  The previous example is considered by Ghigi-Koll\'ar \cite[Example 43]{ghigi_kollar:kaehl_einst_metric_orbif_einst_metric_spher}.  They show that as long as 
\[ 1< \sum_{i=1}^{n+2} \frac{1}{m_i} <  1 + (n+1) \min_{1\le i\le n+2}\left(\frac{1}{m_i}\right)\]
then $(X,\Delta)$ admits a K\"ahler-Einstein metric.   Thus the previous example suggests this condition is strict (our slightly weaker inequality comes from only having a proof that a cscK metric implies semistability rather than polystability).   We remark that Ghigi-Koll\'ar also prove a much more general condition under which a K\"ahler-Einstein Fano manifold with boundary divisor $\Delta$ yields a K\"ahler-Einstein orbifold $(X,\Delta)$ \cite[Theorem 41]{ghigi_kollar:kaehl_einst_metric_orbif_einst_metric_spher}.  It is not the case that this condition is simply the index obstruction, and we have not been able to determine if this condition is related to slope stability or if it is also strict.
\end{rmk}

\section{Orbifold ruled surfaces}

Let $(\Sigma,L)$ be a polarised orbifold Riemann surface and 
$\pi\colon E\to \Sigma$ be an orbifold vector bundle of rank $r$.   Then $\PP(E)$ is itself naturally an orbifold: on a chart $U\to U/G$ of $\Sigma$, the $G$ action on $E|_U$ induces an action on $\PP(E|_U)$ (which is effective as the action on $U$ is) and these give orbifold charts on $\PP(E)$. Suppose that the $G$-action on the fibres $E$ over points of $\Sigma$ with stabiliser group $G$ has distinct eigenvalues, so that $\PP(E)$ has codimension two orbifold
locus and all fibres are finite quotients of $\PP^{r-1}$. The hyperplane bundle
$\O_{\PP(E)}(1)$ is both locally ample and relatively ample, so $L_m:=\O_{\PP(E)}(1)\otimes \pi^*L^m$ is ample for $m$ sufficiently large.

We claim that stability of $\PP(E)$ is connected to stability of the underlying bundle $E$.  Here stability of a bundle is to be taken in the sense of Mumford, so define
\[\mu_E:= \frac{\deg E}{\rk E}\] where the degree is taken in the orbifold sense. Then
$E$ is defined to be stable if for all orbifold bundles $F$ with a proper injection
$F\subset E$ we have $\mu_F<\mu_E$.

Now if $F\subset E$ then $\PP(F)$ is a suborbifold of $\PP(E)$.
Using $\pi_* \O_{\PP(E)}(k) = S^k E^*$ one can use orbifold Riemann-Roch to compute the slope of each in exactly the same way as in the manifold case \cite[Section 5.4]{ross_thomas:06:obstr_to_exist_const_scalar}.
The upshot is that the Seshadri constant of $\PP(F)$ is $\epsilon_{\orb}(\PP(F))=1$ and
\[\mu_1(\O_{\PP(F)})- \mu(\PP(E)) = C(\mu_E-\mu_F)\Big(rm+(r-1)\mu(\Sigma)-r\mu_E\Big)\] for some $C>0$, where $\mu(\Sigma)=-\deg K_\orb/2\deg L$ is the orbifold slope of $(\Sigma,L)$.
% Assume for simplicity that $\Sigma$ has positive 
% orbifold canonical bundle (e.g.\ if the underlying space has genus at least
% 1) so $\mu(\Sigma)<0$.
The term inside the last set of brackets is positive for any $m$ sufficiently positive
that $L_m$ is ample (it is essentially the volume of $(\PP(E),L_m)$).
Therefore if $E$ is unstable as an orbifold vector bundle then $(\PP(E),L_m)$
is slope unstable as an orbifold. This result also generalises to higher dimensional base as long as one works near the adiabatic limit of sufficiently large $m$, just as in the manifold case.

If $E$ is polystable then $\PP(E)$ carries an orbifold cscK metric; see for example 
\cite{rollin_singer:05:non_minim_scalar_flat_kaehl}. We therefore get a (partial) converse
-- for strictly unstable bundles, $(\PP(E),L_m)$ does \emph{not} carry an orbifold cscK metric for any $m$. (The discrepancy lies in strictly semistable, but not polystable, bundles.)

In fact Rollin and Singer phrase their results in terms of parabolic bundles, but there is a complete correspondence between orbifold bundles $E$ on $\Sigma$
and parabolic vector bundles $E'$ on the underlying space of $\Sigma$. In the notation
of Section \ref{Qdiv}, the bundle $E'$ is the pushdown of $E$ from the orbifold to its underlying space; this is therefore the vector bundle analogue of rounding down of $\Q$-divisors in the line bundle
case. The information lost is then encoded via the parabolic structure on $E'$ at each of the orbifold points $x$, with rational weights of the form $p_j/\!\ord(x)$ for $p_j<\ord(x)$ corresponding to the weights of the action on $E_x$. See for example \cite[Section 5]{furuta_steer:92:seifer_fibred_homol_3_spher}.
Moreover this correspondence preserves subobjects and their degrees, where the  parabolic degree of $E'$ is defined as \[\operatorname{pardeg} E' = \deg E' + \sum_{x,j} m_{x,j}\frac
{p_j}{\ord(x)}\,.\]
Here the sum is over all orbifold points $x$,  and if the parabolic structure over $x$ is given by the flag $F_0\subset F_1\subset \ldots \subset E'_x$ then $m_{x,j}= \dim F_j/F_{j+1}$.  Thus orbifold stability of $E$ corresponds precisely to the parabolic stability of $E'$.

Rollin and Singer \cite{rollin_singer:05:non_minim_scalar_flat_kaehl} use such orbifold
cscK metrics as a starting point to produce ordinary cscK metrics (with zero scalar curvature,
in fact) on small blow ups of the orbifolds $\PP(E)$, using a gluing method. Our results
suggest that if $E$ is unstable, destabilised by $F$, then one should be able to slope destabilise such blow ups using the pullback (or proper transform) of $\PP(F)$.

\section{Slope stability of canonically polarised orbifolds}

By the orbifold version of the Aubin-Yau theorems, orbifolds which have positive or trivial
canonical bundle admit orbifold K\"ahler-Einstein metrics. Therefore by Corollary \ref{cor}
they are K-semistable, and so by Theorem \ref{slopeK} are also slope semistable. In fact
this can be proved directly. That is, suppose
that $(X,L)$ is a polarised orbifold and either
\begin{enumerate}
\item $K_{\orb}$ is numerically trivial and $L$ is arbitrary \emph{or }
\item $L=K_{\orb}$.
\end{enumerate}
Then $(X,L)$ is slope stable.  The proof is the same as the manifold case (see \cite[Theorem
8.4]{ross_thomas:07:study_hilber_mumfor_criter_for} or
\cite[Theorem 5.4]{ross_thomas:06:obstr_to_exist_const_scalar}),
with $K_X$ replaced by $K_{\orb}$, so we do not repeat it here.

%%%%%%%%%%%%%%%%%%%%%%%%%%%%%%%%%%%%%%%%%%%%%%%%%%%%%%%%%%%%%%%%%%%%%%%%

\bibliography{biblio2}
\end{document}